\documentclass[reqno]{amsart}
\usepackage{amssymb}
\usepackage{amsmath}
\usepackage{amsfonts}

\newtheorem{theorem}{Theorem}[section]
\theoremstyle{plain}

\newtheorem{lemma}{Lemma}[section]

\numberwithin{equation}{section}

\begin{document}
\title[A unique continuation result]{A unique continuation result for the
plate equation and an application}
\author{Zehra Arat}
\address{{\small Department of Mathematics,} {\small Faculty of Science,
Hacettepe University, Beytepe 06800}, {\small Ankara, Turkey}}
\email{zarat@hacettepe.edu.tr}
\author{Azer Khanmamedov}
\address{{\small Department of Mathematics,} {\small Faculty of Science,
Hacettepe University, Beytepe 06800}, {\small Ankara, Turkey}}
\email{azer@hacettepe.edu.tr}
\author{Sema Simsek}
\address{{\small Department of Mathematics,} {\small Faculty of Science,
Hacettepe University, Beytepe 06800}, {\small Ankara, Turkey}}
\email{semasimsek@hacettepe.edu.tr}

\begin{abstract}
In this paper, we prove the unique continuation property for the weak
solution of the plate equation with non-smooth coefficients. Then, we apply
this result to study the global attractor for the semilinear plate equation
with a localized damping.
\end{abstract}

\maketitle

\section{\protect\bigskip Introduction}

This paper is devoted to investigation of unique continuation property for
the following plate equation%
\begin{equation}
u_{tt}\left( t,x\right) +\Delta ^{2}u\left( t,x\right)
+\sum\limits_{\left\vert \alpha \right\vert \leq 2}q_{\alpha }\left(
t\right) \partial _{x}^{\alpha }u\left( t,x\right) +p\left( t,x_{1}\right)
u\left( t,x\right) =0\text{, }\left( t,x\right) \in
\mathbb{R}
\times
\mathbb{R}
^{n}\text{,}  \tag{1.1}
\end{equation}%
where $\alpha =\left( \alpha _{1},...,\alpha _{n}\right) $ is a multi-index
with $\left\vert \alpha \right\vert =\alpha _{1}+...+\alpha _{n}$, and $%
x=\left( x_{1},...,x_{n}\right) $, $\partial _{x}^{\alpha }=\left( \frac{%
\partial }{\partial x_{1}}\right) ^{\alpha _{1}}...\left( \frac{\partial }{%
\partial x_{n}}\right) ^{\alpha _{n}}$.

The unique continuation property has been intensively studied for a long
time due to the significant role in some subjects, particularly, inverse
problems, control theory and stability. The first well-known result obtained
in this area is the classical Holmgren uniqueness theorem in which the
unique continuation property of solutions for elliptic partial differential
equation with analytic coefficients was proved (see for example [1], [2]).

Over the last few decades, unique continuation results for hyperbolic-like
equations have been drawing more attention. To the best of our knowledge,
one of the first unique continuation results for wave equations was obtained
in a paper by Ruiz [3] which deals with the unique continuation property of
weak $L^{2}$- solutions. This result was successfully applied in the study
of long time dynamics of wave equations, such as exponential decay of
solutions (see for example [4], [5], [6]) and existence of global attractors
(see for example [7], [8]). For the plate equations, an important work was
done by Kim [9] where the Euler-Bernoulli equation with non-smooth
coefficients was considered and the unique continuation result for solutions
in $L^{2}\left( 0,T;H^{3}\left( \Omega \right) \right) \cap W^{1,2}\left(
0,T;H^{1}\left( \Omega \right) \right) $ was obtained. Later, Isakov [10]
established unique continuation property of the strong solutions for wider
class of plate equations with bounded measurable coefficients.

In the articles mentioned above, unique continuation property was proved by
using Carleman estimates. The main goal of this paper is to give simpler
proof of unique continuation property for the weaker, precisely $L^{1}$
(with respect to the variable $x$), solutions of (1.1). To this end, using
the Fourier series expansion, we reduce the considered problem to the
estimation of solutions for an ordinary differential equation and then,
applying the representation formula for the second order differential
equation, we prove the desired result which is as follows:

\begin{theorem}
Assume that $p\in L^{\infty }\left(
\mathbb{R}
;L_{loc}^{1}\left(
\mathbb{R}
\right) \right) $ and $q_{\alpha }\in L^{\infty }\left(
\mathbb{R}
\right) $, for\ $\left\vert \alpha \right\vert \leq 2$. Let $u\in C\left(
\mathbb{R}
;L_{loc}^{1}\left(
\mathbb{R}
^{n}\right) \right) \cap L^{\infty }\left(
\mathbb{R}
;L_{loc}^{1}\left(
\mathbb{R}
^{n}\right) \right) $, with $pu\in L^{\infty }\left(
\mathbb{R}
;L_{loc}^{1}\left(
\mathbb{R}
^{n}\right) \right) $, be a weak solution of (1.1). If%
\begin{equation}
u\left( t,x\right) =0\text{, \ \ }t\in
\mathbb{R}
\text{, }\left\vert x\right\vert \geq r\text{ \ \ }  \tag{1.2}
\end{equation}%
holds for some $r>0$, then \newline
\begin{equation*}
u\left( t,x\right) =0\text{,}\ \ \ \text{a.e. in }%
\mathbb{R}
^{n}\text{,}
\end{equation*}%
for all $t\in
\mathbb{R}
$.
\end{theorem}

Theorem 1.1 is important from the point of view of application to the
semilinear plate equation with a localized damping and non-smooth nonlinear
terms to study the long time behaviour of solutions. We apply the main
result to the $\omega $-limit and $\alpha $-limit sets of the semilinear
plate equation (see Section 3) to show that they are subsets of the set of
stationary points.

The paper is organized as follows. In the next section, we give some
auxiliary lemmas and then prove the main result. In Section 3, we firstly
prove the asymptotic compactness of the semigroup generated by the
semilinear plate equation and then apply Theorem 1.1 to study the global
attractor of this semigroup.

\section{Proof of the main result}

We start with the following lemmas.

\begin{lemma}
Let $y\in C\left(
\mathbb{R}
;%
\mathbb{C}
\right) \cap L^{\infty }\left(
\mathbb{R}
;%
\mathbb{C}
\right) $ be a weak solution of the equation%
\begin{equation}
y^{\prime \prime }\left( t\right) +\left( A\left( t\right) -c^{2}\right)
y\left( t\right) =B\left( t\right) \text{, \ \ }t\in
\mathbb{R}
\text{,}  \tag{2.1}
\end{equation}%
where the constant $c\in
\mathbb{C}
$\ satisfies Re$(c)>0$. Assume that $A\in L^{\infty }\left(
\mathbb{R}
;%
\mathbb{C}
\right) $ is a function such that%
\begin{equation}
\text{\ }\left\Vert A\right\Vert _{L^{\infty }\left(
\mathbb{R}
;%
\mathbb{C}
\right) }<\left\vert c\right\vert \text{Re}(c)\text{,}  \tag{2.2}
\end{equation}%
and%
\begin{equation*}
B\in L^{\infty }\left(
\mathbb{R}
;%
\mathbb{C}
\right) .
\end{equation*}%
Then
\begin{equation*}
\left\vert y\left( t\right) \right\vert \leq \frac{1}{\left\vert
c\right\vert \text{Re}(c)-\left\Vert A\right\Vert _{L^{\infty }\left(
\mathbb{R}
;%
\mathbb{C}
\right) }}\left\Vert B\right\Vert _{L^{\infty }\left(
\mathbb{R}
;%
\mathbb{C}
\right) },\text{ }\forall t\in
\mathbb{R}
\text{.}
\end{equation*}
\end{lemma}

\begin{proof}
Firstly, we write (2.1) in the form%
\begin{equation*}
y^{\prime \prime }\left( t\right) -c^{2}y\left( t\right) =-A\left( t\right)
y\left( t\right) +B\left( t\right) \text{.}
\end{equation*}%
If we solve the following homogeneous part%
\begin{equation*}
y^{\prime \prime }\left( t\right) -c^{2}y\left( t\right) =0\text{,}
\end{equation*}%
we get the linearly independent solutions%
\begin{equation*}
\varphi _{1}\left( t\right) =e^{ct}\text{, }\varphi _{2}\left( t\right)
=e^{-ct}\text{.}
\end{equation*}%
Then, by method of variation of parameters, we have the following
representation formula (see for example [11, p.162-166]) for the solution of
(2.1)%
\begin{equation*}
y\left( t\right) =\left[ \int\limits_{t_{0}}^{t}\frac{\varphi _{2}\left(
s\right) \left( A\left( s\right) y\left( s\right) -B\left( s\right) \right)
}{W_{\varphi _{1},\varphi _{2}}\left( s\right) }ds+c_{1}\right] \varphi
_{1}\left( t\right)
\end{equation*}%
\begin{equation*}
+\left[ c_{2}-\int\limits_{t_{0}}^{t}\frac{\varphi _{1}\left( s\right)
\left( A\left( s\right) y\left( s\right) -B\left( s\right) \right) }{%
W_{\varphi _{1},\varphi _{2}}\left( s\right) }ds\right] \varphi _{2}\left(
t\right) \text{,}
\end{equation*}%
where $t_{0}\in
\mathbb{R}
$ and $W_{\varphi _{1},\varphi _{2}}$ is the Wronskian of $\varphi
_{1},\varphi _{2}$. Since%
\begin{equation*}
W_{\varphi _{1},\varphi _{2}}\left( t\right) =\left\vert
\begin{array}{cc}
e^{ct} & e^{-ct} \\
ce^{ct} & -ce^{-ct}%
\end{array}%
\right\vert =-2c\text{,}
\end{equation*}%
we have%
\begin{equation*}
y\left( t\right) =c_{1}e^{ct}+c_{2}e^{-ct}
\end{equation*}%
\begin{equation*}
-\frac{1}{2c}e^{ct}\int\limits_{t_{0}}^{t}e^{-cs}\left( A\left( s\right)
y\left( s\right) -B\left( s\right) \right) ds
\end{equation*}%
\begin{equation}
+\frac{1}{2c}e^{-ct}\int\limits_{t_{0}}^{t}e^{cs}\left( A\left( s\right)
y\left( s\right) -B\left( s\right) \right) ds\text{.}  \tag{2.3}
\end{equation}%
If we multiply both sides of (2.3) by $e^{-ct}$, we obtain%
\begin{equation*}
e^{-ct}y\left( t\right) =c_{1}+c_{2}e^{-2ct}
\end{equation*}%
\begin{equation*}
-\frac{1}{2c}\int\limits_{t_{0}}^{t}e^{-cs}\left( A\left( s\right) y\left(
s\right) -B\left( s\right) \right) ds+\frac{1}{2c}e^{-2ct}\int%
\limits_{t_{0}}^{t}e^{cs}\left( A\left( s\right) y\left( s\right) -B\left(
s\right) \right) ds\text{.}
\end{equation*}%
Passing to limit as $t\rightarrow \infty $, considering that $y\in C\left(
\mathbb{R}
;%
\mathbb{C}
\right) \cap L^{\infty }\left(
\mathbb{R}
;%
\mathbb{C}
\right) $ and Re$(c)>0$, we get%
\begin{equation}
c_{1}-\frac{1}{2c}\int\limits_{t_{0}}^{\infty }e^{-cs}\left( A\left(
s\right) y\left( s\right) -B\left( s\right) \right) ds=0\text{.}  \tag{2.4}
\end{equation}%
On the other hand, if we multiply both sides of (2.3) by $e^{ct}$, we find%
\begin{equation*}
e^{ct}y\left( t\right) =c_{2}+c_{1}e^{2ct}
\end{equation*}%
\begin{equation*}
+\frac{1}{2c}\int\limits_{t_{0}}^{t}e^{cs}\left( A\left( s\right) y\left(
s\right) -B\left( s\right) \right) ds-\frac{1}{2c}e^{2ct}\int%
\limits_{t_{0}}^{t}e^{-cs}\left( A\left( s\right) y\left( s\right) -B\left(
s\right) \right) ds\text{.}
\end{equation*}%
Passing to limit as $t\rightarrow -\infty $, we obtain%
\begin{equation}
c_{2}-\frac{1}{2c}\int\limits_{-\infty }^{t_{0}}e^{cs}\left( A\left(
s\right) y\left( s\right) -B\left( s\right) \right) ds=0\text{.}  \tag{2.5}
\end{equation}%
Now, using (2.3)-(2.5), we have%
\begin{equation*}
y\left( t_{0}\right) =c_{1}e^{ct_{0}}+c_{2}e^{-ct_{0}}
\end{equation*}%
\begin{equation*}
=\frac{e^{ct_{0}}}{2c}\int\limits_{t_{0}}^{\infty }e^{-cs}\left( A\left(
s\right) y\left( s\right) -B\left( s\right) \right) ds+\frac{e^{-ct_{0}}}{2c}%
\int\limits_{-\infty }^{t_{0}}e^{cs}\left( A\left( s\right) y\left( s\right)
-B\left( s\right) \right) ds\text{.}
\end{equation*}%
Since $t_{0}$\bigskip\ was arbitrary, we can write%
\begin{equation*}
y\left( t\right) =\frac{e^{ct}}{2c}\int\limits_{t}^{\infty }e^{-cs}\left(
A\left( s\right) y\left( s\right) -B\left( s\right) \right) ds+\frac{e^{-ct}%
}{2c}\int\limits_{-\infty }^{t}e^{cs}\left( A\left( s\right) y\left(
s\right) -B\left( s\right) \right) ds\text{, \ \ }\forall t\in
\mathbb{R}
\text{.}
\end{equation*}%
Then
\begin{equation*}
\left\vert y\left( t\right) \right\vert \leq \frac{1}{\left\vert
c\right\vert \text{Re}(c)}\left\Vert Ay-B\right\Vert _{L^{\infty }\left(
\mathbb{R}
;%
\mathbb{C}
\right) }=\alpha \left\Vert y\right\Vert _{L^{\infty }\left(
\mathbb{R}
;%
\mathbb{C}
\right) }+\frac{1}{\left\vert c\right\vert \text{Re}(c)}\left\Vert
B\right\Vert _{L^{\infty }\left(
\mathbb{R}
;%
\mathbb{C}
\right) }\text{,}
\end{equation*}%
and consequently
\begin{equation*}
\left\Vert y\right\Vert _{L^{\infty }\left(
\mathbb{R}
;%
\mathbb{C}
\right) }\leq \alpha \left\Vert y\right\Vert _{L^{\infty }\left(
\mathbb{R}
;%
\mathbb{C}
\right) }+\frac{1}{\left\vert c\right\vert \text{Re}(c)}\left\Vert
B\right\Vert _{L^{\infty }\left(
\mathbb{R}
;%
\mathbb{C}
\right) }\text{,}
\end{equation*}%
where $\alpha =\frac{1}{\left\vert c\right\vert \text{Re}(c)}\left\Vert
A\right\Vert _{L^{\infty }\left(
\mathbb{R}
;%
\mathbb{C}
\right) }.$Since $\alpha <1$ (see (2.2)), by the above inequality, we obtain
the claim of lemma.
\end{proof}

\begin{lemma}
Assume that $q_{2}\in L^{\infty }\left(
\mathbb{R}
;L_{loc}^{1}\left(
\mathbb{R}
\right) \right) $ and $q_{j}\in L^{\infty }\left(
\mathbb{R}
\right) $, for $j=0,1$. Let $u\in C\left(
\mathbb{R}
;L_{loc}^{1}\left(
\mathbb{R}
\right) \right) \cap L^{\infty }\left(
\mathbb{R}
;L_{loc}^{1}\left(
\mathbb{R}
\right) \right) ,$ with $q_{2}u\in L^{\infty }\left(
\mathbb{R}
;L_{loc}^{1}\left(
\mathbb{R}
\right) \right) $, be a weak solution of the equation%
\begin{equation*}
u_{tt}\left( t,x\right) +u_{xxxx}\left( t,x\right) +q_{0}\left( t\right)
u_{xx}\left( t,x\right) +q_{1}\left( t\right) u_{x}\left( t,x\right)
\end{equation*}%
\begin{equation}
+q_{2}\left( t,x\right) u\left( t,x\right) =0\text{, \ }\left( t,x\right)
\in
\mathbb{R}
\times
\mathbb{R}
\text{.}  \tag{2.6}
\end{equation}%
If
\begin{equation}
u\left( t,x\right) =0\text{,}\ \ \ t\in
\mathbb{R}
\text{, \ }\left\vert x\right\vert \geq r\text{ }  \tag{2.7}
\end{equation}%
holds for some $r>0$, then%
\begin{equation*}
u\left( t,x\right) =0\text{,}\ \ \ \text{a.e. in }%
\mathbb{R}
\text{,}
\end{equation*}%
for all $t\in
\mathbb{R}
$.
\end{lemma}

\begin{proof}
Testing the equation (2.6)$_{\text{ }}$with $e^{\lambda x}e^{i\frac{\pi }{2r}%
\left( x+r\right) k}$ in $%
\mathbb{R}
$, where $\lambda >0$ and $k\in
\mathbb{Z}
^{+}$, and considering (2.7), we get%
\begin{equation*}
\frac{d^{2}}{dt^{2}}\int\limits_{-r}^{r}u\left( t,x\right) e^{\lambda x}e^{i%
\frac{\pi }{2r}\left( x+r\right) k}dx+\left( \lambda +\frac{i\pi k}{2r}%
\right) ^{4}\int\limits_{-r}^{r}u\left( t,x\right) e^{\lambda x}e^{i\frac{%
\pi }{2r}\left( x+r\right) k}dx
\end{equation*}%
\begin{equation*}
+\left( \lambda +\frac{i\pi k}{2r}\right)
^{2}q_{0}(t)\int\limits_{-r}^{r}u\left( t,x\right) e^{\lambda x}e^{i\frac{%
\pi }{2r}\left( x+r\right) k}dx
\end{equation*}%
\begin{equation*}
-\left( \lambda +\frac{i\pi k}{2r}\right)
q_{1}(t)\int\limits_{-r}^{r}u\left( t,x\right) e^{\lambda x}e^{i\frac{\pi }{%
2r}\left( x+r\right) k}dx
\end{equation*}%
\begin{equation*}
+\int\limits_{-r}^{r}q_{2}\left( t,x\right) u\left( t,x\right) e^{\lambda
x}e^{i\frac{\pi }{2r}\left( x+r\right) k}dx=0\text{.}
\end{equation*}%
If we define%
\begin{equation*}
y_{k,\lambda }\left( t\right) :=\int\limits_{-r}^{r}u\left( t,x\right)
e^{\lambda x}e^{i\frac{\pi }{2r}\left( x+r\right) k}dx\text{,}
\end{equation*}%
then we have%
\begin{equation*}
y_{k,\lambda }^{\prime \prime }\left( t\right) +\left[ \left( \lambda +\frac{%
i\pi k}{2r}\right) ^{4}+\left( \lambda +\frac{i\pi k}{2r}\right)
^{2}q_{0}\left( t\right) \right.
\end{equation*}%
\begin{equation*}
-\left. \left( \lambda +\frac{i\pi k}{2r}\right) q_{1}\left( t\right) \right]
y_{k,\lambda }\left( t\right) =-\int\limits_{-r}^{r}q_{2}\left( t,x\right)
u\left( t,x\right) e^{\lambda x}e^{i\frac{\pi }{2r}\left( x+r\right) k}dx%
\text{.}
\end{equation*}%
Now, denote
\begin{equation*}
A_{k,\lambda }\left( t\right) :=\left( \lambda +\frac{i\pi k}{2r}\right)
^{2}q_{0}\left( t\right) -\left( \lambda +\frac{i\pi k}{2r}\right) q_{1}(t),
\end{equation*}%
\begin{equation*}
B_{k,\lambda }\left( t\right) =-\int\limits_{-r}^{r}q_{2}\left( t,x\right)
u\left( t,x\right) e^{\lambda x}e^{i\frac{\pi }{2r}\left( x+r\right) k}dx,
\end{equation*}%
and%
\begin{equation*}
c_{k,\lambda }:=-i\left( \lambda +\frac{i\pi k}{2r}\right) ^{2}\text{.}
\end{equation*}%
Then, there exists $\lambda _{0}>0$ such that $A_{k,\lambda }\left( t\right)
$, $B_{k,\lambda }\left( t\right) $ and $c_{k,\lambda }$ satisfy the
conditions of Lemma 2.1 for all $\lambda \geq \lambda _{0}$ and $k\in
\mathbb{Z}
^{+}$. Hence, applying Lemma 2.1, we obtain%
\begin{equation*}
\left\vert y_{k,\lambda }\left( t\right) \right\vert \leq \frac{\widetilde{c}%
_{1}}{\lambda k^{3}}\left\Vert B_{k,\lambda }\right\Vert _{L^{\infty }\left(
\mathbb{R}
;%
\mathbb{C}
\right) }\text{,}\ \ \ \forall t\in
\mathbb{R}
\text{, }\forall \lambda \geq \lambda _{0}\text{, \ }\forall k\in
\mathbb{Z}
^{+}\text{,}
\end{equation*}%
which, together with (2.7), give us%
\begin{equation*}
\sum_{k=1}^{\infty }k^{4}\left\vert y_{k,\lambda }\left( t\right)
\right\vert ^{2}\leq \frac{\widetilde{c}_{2}}{\lambda ^{2}}\left( \underset{%
t\in
\mathbb{R}
}{ess\sup }\int\limits_{-r}^{r}\left\vert q_{2}\left( t,x\right) u\left(
t,x\right) \right\vert e^{\lambda x}dx\right) ^{2}<\infty ,\text{ }\forall
t\in
\mathbb{R}
\text{, }\forall \lambda \geq \lambda _{0}\text{,}
\end{equation*}%
and by the definition of $y_{k,\lambda }\left( t\right) $, we find%
\begin{equation*}
\sum_{k=1}^{\infty }k^{4}\left\vert \int\limits_{-r}^{r}u\left( t,x\right)
e^{\lambda x}\sin \left( \frac{\pi }{2r}\left( x+r\right) k\right)
dx\right\vert ^{2}
\end{equation*}%
\begin{equation*}
\leq \frac{\widetilde{c}_{2}}{\lambda ^{2}}\left( \underset{t\in
\mathbb{R}
}{ess\sup }\int\limits_{-r}^{r}\left\vert q_{2}\left( t,x\right) u\left(
t,x\right) \right\vert e^{\lambda x}dx\right) ^{2}<\infty ,\text{ }\forall
t\in
\mathbb{R}
\text{, }\forall \lambda \geq \lambda _{0}\text{.}
\end{equation*}%
Since $\left\{ \frac{1}{\sqrt{r}}\sin \left( \frac{\pi }{2r}\left(
x+r\right) k\right) \right\} _{k=1}^{\infty }$ is orthonormal basis in $%
L^{2}\left( -r,r\right) $ consisting of the eigenfunctions of the operator $-%
\frac{\partial ^{2}}{\partial x^{2}}$ in $L^{2}\left( -r,r\right) $ with the
domain $H^{2}\left( -r,r\right) \cap H_{0}^{1}\left( -r,r\right) $, by the
last inequality, we find that $u\in L^{\infty }\left(
\mathbb{R}
;H^{2}\left( -r,r\right) \cap H_{0}^{1}\left( -r,r\right) \right) $, and%
\begin{equation*}
\int\limits_{-r}^{r}\left\vert \frac{\partial ^{2}}{\partial x^{2}}(u\left(
t,x\right) e^{\lambda x})\right\vert ^{2}dx\leq \frac{\widetilde{c}_{2}}{%
\lambda ^{2}}\left( \underset{t\in
\mathbb{R}
}{ess\sup }\int\limits_{-r}^{r}\left\vert q_{2}\left( t,x\right) u\left(
t,x\right) \right\vert e^{\lambda x}dx\right) ^{2},\text{ \ \ }\forall t\in
\mathbb{R}
\text{, \ }\forall \lambda \geq \lambda _{0}\text{.}
\end{equation*}%
Then, considering $H^{2}\left( -r,r\right) \subset L^{\infty }(-r,r),$ we get%
\begin{equation*}
\left\Vert u\left( t,\cdot \right) e^{\lambda \cdot }\right\Vert _{L^{\infty
}(-r,r)}\leq \frac{\widetilde{c}_{3}}{\lambda }\left\Vert u\left( \cdot
,\cdot \right) e^{\lambda \cdot }\right\Vert _{L^{\infty }(%
\mathbb{R}
\times (-r,r))},\text{ \ \ }\forall t\in
\mathbb{R}
\text{, \ }\forall \lambda \geq \lambda _{0}\text{.}
\end{equation*}%
Choosing $\lambda $ large enough in the above inequality, we obtain%
\begin{equation*}
u\left( t,x\right) =0\text{, \ \ a.e. in }%
\mathbb{R}
\text{,}
\end{equation*}%
for all $t\in
\mathbb{R}
$.
\end{proof}

Now, we can prove the main theorem. We use induction on $n$. For $n=1$, we
obtain the result by Lemma 2.2. Now, assume that the claim of \ Theorem 1.1
holds in $\left( n-1\right) $- dimensional (with respect to the space
variable) case. Let $u\left( t,x\right) $ be a weak solution of (1.1),
satisfying (1.2), where $x:=\left( \overset{\_}{x},x_{n}\right) :=\left(
x_{1},x_{2},...,x_{n-1},x_{n}\right) \in
\mathbb{R}
^{n}$. For $\psi \in C_{0}^{\infty }\left(
\mathbb{R}
\right) $, $\psi \equiv 1$ in $\left[ -r,r\right] $ and $\varphi \in
C_{0}^{\infty }\left(
\mathbb{R}
^{n-1}\right) $, by testing the equation (1.1) with the function $\varphi
\left( \overset{\_}{x}\right) \psi \left( x_{n}\right) $ in $%
\mathbb{R}
^{n}$ and considering (1.2), we obtain%
\begin{equation*}
\frac{d^{2}}{dt^{2}}\int\limits_{%
\mathbb{R}
^{n-1}}\varphi \left( \overset{\_}{x}\right) \int\limits_{-r}^{r}u\left(
t,x\right) dx_{n}d\overline{x}+\int\limits_{%
\mathbb{R}
^{n-1}}\Delta _{\overset{\_}{x}}^{2}\varphi \left( \overset{\_}{x}\right)
\int\limits_{-r}^{r}u\left( t,x\right) dx_{n}d\overline{x}
\end{equation*}%
\begin{equation*}
+\sum\limits_{\underset{\alpha _{n}=0}{\left\vert \alpha \right\vert \leq 2}%
}\left( -1\right) ^{\left\vert \alpha \right\vert }q_{\alpha }\left(
t\right) \int\limits_{%
\mathbb{R}
^{n-1}}\partial _{\overline{x}}^{\alpha }\varphi \left( \overset{\_}{x}%
\right) \int\limits_{-r}^{r}u\left( t,x\right) dx_{n}d\overline{x}
\end{equation*}%
\begin{equation}
+\int\limits_{%
\mathbb{R}
^{n-1}}p\left( t,x_{1}\right) \varphi \left( \overset{\_}{x}\right)
\int\limits_{-r}^{r}u\left( t,x\right) dx_{n}d\overline{x}=0  \tag{2.8}
\end{equation}%
where $\Delta _{\overset{\_}{x}}=\sum\limits_{i=1}^{n-1}\frac{\partial ^{2}}{%
\partial x_{i}^{2}}$ and $\partial _{\overline{x}}^{\alpha }=\left( \frac{%
\partial }{\partial x_{1}}\right) ^{\alpha _{1}}...\left( \frac{\partial }{%
\partial x_{n-1}}\right) ^{\alpha _{n-1}}$. If we define%
\begin{equation*}
v_{0}\left( t,\overset{\_}{x}\right) :=\int\limits_{-r}^{r}u\left(
t,x\right) dx_{n}\text{,}
\end{equation*}%
then, by (2.8) and (1.2), $v_{0}\left( t,\overset{\_}{x}\right) $ is a weak
solution of the following $\left( n-1\right) $-dimensional (with respect to
the space variable) problem%
\begin{equation*}
\left\{
\begin{array}{l}
v_{0tt}\left( t,\overset{\_}{x}\right) +\Delta _{\overset{\_}{x}%
}^{2}v_{0}\left( t,\overset{\_}{x}\right) +\sum\limits_{\underset{\alpha
_{n}=0}{\left\vert \alpha \right\vert \leq 2}}q_{\alpha }\left( t\right)
\partial _{\overline{x}}^{\alpha }v_{0}\left( t,\overset{\_}{x}\right)
+p\left( t,x_{1}\right) v_{0}\left( t,\overset{\_}{x}\right) =0\text{, }%
\left( t,\overset{\_}{x}\right) \in
\mathbb{R}
\times
\mathbb{R}
^{n-1}\text{,} \\
v_{0}\left( t,\overline{x}\right) =0\text{, \ }\left\vert \overline{x}%
\right\vert \geq r\text{.}%
\end{array}%
\right.
\end{equation*}%
Using induction assumption, we have%
\begin{equation}
v_{0}\left( t,\overset{\_}{x}\right) =0\text{, \ \ a.e. in }%
\mathbb{R}
^{n-1}\text{,}  \tag{2.9}
\end{equation}%
for all $t\in
\mathbb{R}
$. Now, by testing the equation (1.1) with $x_{n}\varphi \left( \overset{\_}{%
x}\right) \psi \left( x_{n}\right) $ in $%
\mathbb{R}
^{n}$, considering (1.2) and (2.9), we get
\begin{equation*}
\frac{d^{2}}{dt^{2}}\int\limits_{%
\mathbb{R}
^{n-1}}\varphi \left( \overset{\_}{x}\right) \int\limits_{-r}^{r}u\left(
t,x\right) x_{n}dx_{n}d\overline{x}+\int\limits_{%
\mathbb{R}
^{n-1}}\Delta _{\overset{\_}{x}}^{2}\varphi \left( \overset{\_}{x}\right)
\int\limits_{-r}^{r}u\left( t,x\right) x_{n}dx_{n}d\overline{x}
\end{equation*}%
\begin{equation*}
+\sum\limits_{\underset{\alpha _{n}=0}{\left\vert \alpha \right\vert \leq 2}%
}\left( -1\right) ^{\left\vert \alpha \right\vert }q_{\alpha }\left(
t\right) \int\limits_{%
\mathbb{R}
^{n-1}}\partial _{\overline{x}}^{\alpha }\varphi \left( \overset{\_}{x}%
\right) \int\limits_{-r}^{r}u\left( t,x\right) x_{n}dx_{n}d\overline{x}\text{%
.}
\end{equation*}%
\begin{equation}
+\int\limits_{%
\mathbb{R}
^{n-1}}p\left( t,x_{1}\right) \varphi \left( \overset{\_}{x}\right)
\int\limits_{-r}^{r}u\left( t,x\right) x_{n}dx_{n}d\overline{x}=0.
\tag{2.10}
\end{equation}%
If we define%
\begin{equation*}
v_{1}\left( t,\overset{\_}{x}\right) :=\int\limits_{-r}^{r}x_{n}u\left(
t,x\right) dx_{n}\text{,}
\end{equation*}%
then, by (2.10) and (1.3), $v_{1}\left( t,\overset{\_}{x}\right) $ is a weak
solution of the following problem
\begin{equation*}
\left\{
\begin{array}{l}
v_{1tt}\left( t,\overset{\_}{x}\right) +\Delta _{\overset{\_}{x}%
}^{2}v_{1}\left( t,\overset{\_}{x}\right) +\sum\limits_{\underset{\alpha
_{n}=0}{\left\vert \alpha \right\vert \leq 2}}q_{\alpha }\left( t\right)
\partial _{\overline{x}}^{\alpha }v_{1}\left( t,\overset{\_}{x}\right)
+p\left( t,x_{1}\right) v_{1}\left( t,\overset{\_}{x}\right) =0\text{, }%
\left( t,\overset{\_}{x}\right) \in
\mathbb{R}
\times
\mathbb{R}
^{n-1}\text{,} \\
v_{1}\left( t,\overline{x}\right) =0\text{, \ }\left\vert \overline{x}%
\right\vert \geq r\text{ }%
\end{array}%
\right.
\end{equation*}%
and so, by induction assumption,%
\begin{equation}
v_{1}\left( t,\overset{\_}{x}\right) =0\text{, \ \ a.e. in }%
\mathbb{R}
^{n-1}\text{,}  \tag{2.11}
\end{equation}%
for all $t\in
\mathbb{R}
$. Similarly, by testing the equation (1.1) with $x_{n}^{2}\varphi \left(
\overset{\_}{x}\right) \psi \left( x_{n}\right) $ in $%
\mathbb{R}
^{n}$, considering (1.2), (2.9) and (2.11), we find%
\begin{equation*}
\frac{d^{2}}{dt^{2}}\int\limits_{%
\mathbb{R}
^{n-1}}\varphi \left( \overset{\_}{x}\right) \int\limits_{-r}^{r}u\left(
t,x\right) x_{n}^{2}dx_{n}d\overline{x}+\int\limits_{%
\mathbb{R}
^{n-1}}\Delta _{\overline{x}}^{2}\varphi \left( \overset{\_}{x}\right)
\int\limits_{-r}^{r}u\left( t,x\right) x_{n}^{2}dx_{n}d\overline{x}
\end{equation*}%
\begin{equation*}
+\sum\limits_{\underset{\alpha _{n}=0}{\left\vert \alpha \right\vert \leq 2}%
}\left( -1\right) ^{\left\vert \alpha \right\vert }q_{\alpha }\left(
t\right) \int\limits_{%
\mathbb{R}
^{n-1}}\partial _{\overline{x}}^{\alpha }\varphi \left( \overset{\_}{x}%
\right) \int\limits_{-r}^{r}u\left( t,x\right) x_{n}^{2}dx_{n}d\overline{x}
\end{equation*}%
\begin{equation}
+\int\limits_{%
\mathbb{R}
^{n-1}}p\left( t,x_{1}\right) \varphi \left( \overset{\_}{x}\right)
\int\limits_{-r}^{r}u\left( t,x\right) x_{n}^{2}dx_{n}d\overline{x}=0
\tag{2.12}
\end{equation}%
If we define%
\begin{equation*}
v_{2}\left( t,\overset{\_}{x}\right) :=\int\limits_{-r}^{r}x_{n}^{2}u\left(
t,x\right) dx_{n}\text{,}
\end{equation*}%
then by (2.12) and (1.2), $v_{2}\left( t,\overset{\_}{x}\right) $ is a weak
solution of the following problem
\begin{equation*}
\left\{
\begin{array}{l}
v_{2tt}\left( t,\overset{\_}{x}\right) +\Delta _{\overset{\_}{x}%
}^{2}v_{2}\left( t,\overset{\_}{x}\right) +\sum\limits_{\underset{\alpha
_{n}=0}{\left\vert \alpha \right\vert \leq 2}}q_{\alpha }\left( t\right)
\partial _{\overline{x}}^{\alpha }v_{2}\left( t,\overset{\_}{x}\right)
+p\left( t,x_{1}\right) v_{2}\left( t,\overset{\_}{x}\right) =0\text{, }%
\left( t,\overset{\_}{x}\right) \in
\mathbb{R}
\times
\mathbb{R}
^{n-1}\text{,} \\
v_{2}\left( t,\overline{x}\right) =0\text{, \ }\left\vert \overline{x}%
\right\vert \geq r\text{.}%
\end{array}%
\right.
\end{equation*}%
By induction assumption,%
\begin{equation*}
v_{2}\left( t,\overset{\_}{x}\right) =0\text{, \ \ a.e. in }%
\mathbb{R}
^{n-1}\text{,}
\end{equation*}%
for all $t\in
\mathbb{R}
$. Continuing this procedure, we have
\begin{equation*}
v_{m}\left( t,\overset{\_}{x}\right) =0\text{, \ \ a.e. in }%
\mathbb{R}
^{n-1}\text{, }m=0,1,...
\end{equation*}%
for all $t\in
\mathbb{R}
$, where
\begin{equation*}
v_{m}\left( t,\overset{\_}{x}\right) :=\int\limits_{-r}^{r}x_{n}^{m}u\left(
t,x\right) dx_{n}\text{. }
\end{equation*}%
Since $u\in C\left(
\mathbb{R}
,L^{1}\left(
\mathbb{R}
^{n}\right) \right) $ and the set of monomials is dense in $C\left[ -r,r%
\right] $, we obtain
\begin{equation*}
\int\limits_{-r}^{r}\varphi \left( x_{n}\right) u\left( t,x\right) dx_{n}=0%
\text{, \ \ a.e. in }%
\mathbb{R}
^{n-1}\text{, }
\end{equation*}%
for all $t\in
\mathbb{R}
$ and $\varphi \in C\left[ -r,r\right] $, which gives us
\begin{equation*}
u\left( t,x\right) =0\text{, \ \ a.e. in }%
\mathbb{R}
^{n}\text{,}
\end{equation*}%
for all $t\in
\mathbb{R}
$.\newline

\section{An application: Global attractor for the semilinear plate equation}

Let $\Omega \subset
\mathbb{R}
^{n}$, $n\geq 1$, be a bounded domain with sufficiently smooth boundary $%
\partial \Omega $. We consider the following initial boundary value problem:%
\begin{equation*}
u_{tt}\left( t,x\right) +\Delta ^{2}u\left( t,x\right) +\alpha \left(
x\right) u_{t}\left( t,x\right) -f_{1}\left( \left\Vert \nabla u\left(
t\right) \right\Vert _{L^{2}\left( \Omega \right) }\right) \Delta u\left(
t,x\right)
\end{equation*}%
\begin{equation}
+f_{2}\left( \left\Vert u\left( t\right) \right\Vert _{L^{2}\left( \Omega
\right) }\right) u\left( t,x\right) +\beta (x_{1})u\left( t,x\right) =0\text{%
, }(t,x)\in (0,\infty )\times \Omega \text{,}  \tag{3.1}
\end{equation}%
\begin{equation}
u\left( t,x\right) =\frac{\partial }{\partial \nu }u\left( t,x\right) =0%
\text{, \ \ }\left( t,x\right) \in \left( 0,\infty \right) \times \partial
\Omega \text{,\ \ \ \ \ \ \ \ \ \ \ \ \ \ \ \ \ \ \ \ \ \ }  \tag{3.2}
\end{equation}%
\begin{equation}
u\left( 0,x\right) =u_{0}\left( x\right) \text{, \ }\ \ \ \ u_{t}\left(
0,x\right) =u_{1}\left( x\right) \text{, \ \ }x\in \Omega \text{, \ \ \ \ \
\ \ \ \ \ \ \ \ \ \ \ \ \ \ \ }  \tag{3.3}
\end{equation}%
where $\nu $ is outer unit normal vector, the functions $\alpha $, $\beta $,
$f_{1}$ and $f_{2}$ satisfy the following conditions:%
\begin{equation}
\alpha \in L^{\infty }\left( \Omega \right) ,\text{ }\alpha \left( \cdot
\right) \geq 0\text{, \ \ a.e. in }\Omega \text{,}  \tag{3.4}
\end{equation}%
\begin{equation}
\alpha \left( \cdot \right) \geq \alpha _{0}>0\text{, \ \ a.e. in }\omega
\subset \Omega \text{, where }\omega \text{ is a neighbourhood of }\partial
\Omega \text{,}  \tag{3.5}
\end{equation}%
\begin{equation}
\beta \in L^{1}\left( \Omega _{1}\right) \text{, }\beta \left( \cdot \right)
\geq 0\text{, \ a.e. in }\Omega _{1}\text{, where }\Omega _{1}=\{x_{1}\in
\mathbb{R}
\text{: }\exists y\in
\mathbb{R}
^{n-1}\text{, }(x_{1},y)\in \Omega \}  \tag{3.6}
\end{equation}%
\begin{equation}
f_{1},f_{2}\in C^{0,1}\left(
\mathbb{R}
^{+}\right) \text{ \ and \ }\liminf\limits_{s\rightarrow \infty }f_{1}\left(
s\right) \geq 0\text{, }\liminf\limits_{s\rightarrow \infty }f_{2}\left(
s\right) \geq 0\text{.}  \tag{3.7}
\end{equation}%
In particular case when $f_{1}\left( s\right) =s^{2}-c$, $f_{2}\equiv 0$ and
$\beta \equiv 0$, where $c$ is a constant, equation (3.1) becomes Berger
equation (see [12]).

Now, denoting $A$\ $(w_{{\small 1}},$ $w_{{\small 2}})=(w_{{\small 2}},$ $%
-\Delta ^{{\small 2}}w_{{\small 1}}-\alpha (\cdot )w_{{\small 2}})$, $%
D(A)=H_{0}^{2}\left( \Omega \right) \times L^{2}\left( \Omega \right) $ and $%
\Phi (v(x))=f_{1}\left( \left\Vert \nabla v\right\Vert _{L^{2}\left( \Omega
\right) }\right) \Delta v\left( x\right) -f_{2}\left( \left\Vert
v\right\Vert _{L^{2}\left( \Omega \right) }\right) v\left( x\right) -\beta
(x_{1})v\left( x\right) $, we can reduce (3.1)-(3.3) to the problem%
\begin{equation}
\left\{
\begin{array}{c}
\frac{d}{dt}\left( u,u_{t}\right) =A(u,u_{t})+(0,\Phi (u)) \\
(u(0),u_{t}(0))=\left( u_{0},u_{1}\right)%
\end{array}%
\right.  \tag{3.8}
\end{equation}%
in $L^{2}\left( \Omega \right) \times H^{-2}\left( \Omega \right) $. Since $%
A $ is isomorphism between $H_{0}^{2}\left( \Omega \right) \times
L^{2}\left( \Omega \right) $ and $L^{2}\left( \Omega \right) \times
H^{-2}\left( \Omega \right) $, also between $\left( H^{4}\left( \Omega
\right) \cap H_{0}^{2}\left( \Omega \right) \right) \times H_{0}^{2}\left(
\Omega \right) $ and $H_{0}^{2}\left( \Omega \right) \times L^{2}\left(
\Omega \right) $, defining the equivalent norm $\left\Vert A^{-1}(w_{{\small %
1}},w_{{\small 2}})\right\Vert _{H_{0}^{2}\left( \Omega \right) \times
L^{2}\left( \Omega \right) }$ in $L^{2}\left( \Omega \right) \times
H^{-2}\left( \Omega \right) $, it is easy to verify that $A$ is maximal
dissipative operator in $L^{2}\left( \Omega \right) \times H^{-2}\left(
\Omega \right) $. Hence, $\left\{ e^{tA}\right\} _{t\geq 0}$ is a linear \
continuous semigroup in $H_{0}^{2}\left( \Omega \right) \times L^{2}\left(
\Omega \right) $ and $L^{2}\left( \Omega \right) \times H^{-2}\left( \Omega
\right) $, so, by the interpolation theorem, also in $H_{0}^{1}\left( \Omega
\right) \times H^{-1}\left( \Omega \right) $. Since, the nonlinear operator $%
\Phi :H_{0}^{1}\left( \Omega \right) \rightarrow H^{-1}\left( \Omega \right)
$ satisfies Lipschitz condition%
\begin{equation*}
\left\Vert \Phi (v_{1})-\Phi (v_{2})\right\Vert _{H^{-1}\left( \Omega
\right) }\leq c(r)\left\Vert v_{1}-v_{2}\right\Vert _{H^{1}\left( \Omega
\right) }\text{, \ \ }\forall v_{1},v_{2}\in H_{0}^{2}\left( \Omega \right)
\text{,}
\end{equation*}%
where $c:R_{+}\rightarrow R_{+}$ is a nondecreasing function and $r=\max
\left\{ \left\Vert v_{1}\right\Vert _{H^{2}\left( \Omega \right)
},\left\Vert v_{2}\right\Vert _{H^{2}\left( \Omega \right) }\right\} $, by
the semigroup theory, for every $\left( u_{0},u_{1}\right) \in $ $H_{0}^{1}\left(
\Omega \right) \times H^{-1}\left( \Omega \right) ,$ the problem (3.8) has a
unique local weak solution $(u,u_{t})\in C\left( [0,T_{\max
});H_{0}^{1}\left( \Omega \right) \times H^{-1}\left( \Omega \right) \right)
$. Moreover, if $\left( u_{0},u_{1}\right) \in \left( H^{3}\left( \Omega
\right) \cap H_{0}^{2}\left( \Omega \right) \right) \times H_{0}^{1}\left(
\Omega \right) $, then $(u,u_{t})$ is a strong solution of (3.8) and
consequently of (3.1)-(3.3), from the class $C\left( [0,T_{\max });\left(
H^{3}\left( \Omega \right) \cap H_{0}^{2}\left( \Omega \right) \right)
\times H_{0}^{1}\left( \Omega \right) \right) $.

Let $u\in C\left( [0,T_{\max });H^{3}\left( \Omega \right) \cap
H_{0}^{2}\left( \Omega \right) \right) \cap C^{1}\left( [0,T_{\max
});H_{0}^{1}\left( \Omega \right) \right) $ be local strong solution of
(3.1)-(3.3). Testing (3.1) by $u_{t}$, we get
\begin{equation*}
E\left( u\left( t\right) ,u_{t}\left( t\right) \right)
+\int\limits_{s}^{t}\int\limits_{\Omega }\alpha \left( x\right) \left\vert
u_{t}\left( \tau ,x\right) \right\vert ^{2}dxd\tau
\end{equation*}%
\begin{equation}
\leq E\left( u\left( s\right) ,u_{t}\left( s\right) \right) \text{, \ \ \ }%
t\geq s\geq 0\text{,}  \tag{3.9}
\end{equation}%
where $E\left( u\left( t\right) ,u_{t}\left( t\right) \right) =\frac{1}{2}%
\left( \left\Vert u_{t}\left( t\right) \right\Vert _{L^{2}\left( \Omega
\right) }^{2}+\left\Vert \Delta u\left( t\right) \right\Vert _{L^{2}\left(
\Omega \right) }^{2}\right) +F_{1}\left( \left\Vert \nabla u\left( t\right)
\right\Vert _{L^{2}\left( \Omega \right) }^{2}\right) $\newline
$+F_{2}\left( \left\Vert u\left( t\right) \right\Vert _{L^{2}\left( \Omega
\right) }^{2}\right) +\frac{1}{2}\int\limits_{\Omega }\beta \left(
x_{1}\right) \left\vert u\left( t,x\right) \right\vert ^{2}dx$, $F_{1}\left(
z\right) =\int\limits_{0}^{z}f_{1}\left( \sqrt{\tau }\right) d\tau $ and $%
F_{2}\left( z\right) =\int\limits_{0}^{z}f_{2}\left( \sqrt{\tau }\right)
d\tau $. Considering (3.4), (3.6) and (3.7) in (3.9), we obtain%
\begin{equation}
\left\Vert \left( u\left( t\right) ,u_{t}\left( t\right) \right) \right\Vert
_{H_{0}^{2}\left( \Omega \right) \times L^{2}\left( \Omega \right) }\leq
c\left( \left\Vert \left( u_{0},u_{1}\right) \right\Vert _{H_{0}^{2}\left(
\Omega \right) \times L^{2}\left( \Omega \right) }\right) \text{, \ \ \ }%
t\geq 0\text{,}  \tag{3.10}
\end{equation}%
where $c:%
\mathbb{R}
^{+}\rightarrow $ $%
\mathbb{R}
^{+}$ is a nondecreasing function. The last inequality yields that the local
solution $u$ can be extended to $[0,\infty )$.

Now, let $v,w\in C\left( [0,\infty );H^{3}\left( \Omega \right) \cap
H_{0}^{2}\left( \Omega \right) \right) \cap C^{1}\left( [0,\infty
);H_{0}^{1}\left( \Omega \right) \right) $ be strong solutions of
(3.1)-(3.3) with initial data $\left( v_{0},v_{1}\right) \in \left(
H^{3}\left( \Omega \right) \cap H_{0}^{2}\left( \Omega \right) \right)
\times H_{0}^{1}\left( \Omega \right) $ and $\left( w_{0},w_{1}\right) \in
\left( H^{3}\left( \Omega \right) \cap H_{0}^{2}\left( \Omega \right)
\right) \times H_{0}^{1}\left( \Omega \right) $. Putting $v$ and $w$ instead
of $u$ in (3.1), subtracting the equations and testing the obtained equation
by $\left( v_{t}-w_{t}\right) $, we find%
\begin{equation*}
\left\Vert v(t)-w(t)\right\Vert _{H_{0}^{2}(\Omega )}+\left\Vert
v_{t}(t)-w_{t}(t)\right\Vert _{L^{2}(\Omega )}\leq
\end{equation*}%
\begin{equation*}
\leq \widetilde{c}(T,\widetilde{r})\left( \left\Vert v(0)-w(0)\right\Vert
_{H_{0}^{2}(\Omega )}+\left\Vert v_{t}(0)-w_{t}(0)\right\Vert _{L^{2}(\Omega
)}\right) \text{, \ }\forall t\in \lbrack 0,T]\text{,}
\end{equation*}%
where $\widetilde{c}:R_{+}\times R_{+}\rightarrow R_{+}$ \ is a
nondecreasing function with respect to each variable and $\widetilde{r}=\max
\left\{ \left\Vert (v_{0},v_{1})\right\Vert _{H_{0}^{2}\left( \Omega \right)
\times L^{2}\left( \Omega \right) },\left\Vert (w_{0},w_{1})\right\Vert
_{H_{0}^{2}\left( \Omega \right) \times L^{2}\left( \Omega \right) }\right\}
$. The last inequality, together with (3.10), implies that for every $\left(
u_{0},u_{1}\right) \in H_{0}^{2}\left( \Omega \right) \times L^{2}\left(
\Omega \right) $, problem (3.1)-(3.3) has a unique weak solution $u\in
C\left( [0,\infty );H_{0}^{2}\left( \Omega \right) \right) \cap C^{1}\left(
[0,\infty );L^{2}\left( \Omega \right) \right) $, which depends continuously on the 
initial data. Therefore, the problem
(3.1)-(3.3) generates a strongly continuous semigroup $\left\{ S\left(
t\right) \right\} _{t\geq 0}$ in $H_{0}^{2}\left( \Omega \right) \times
L^{2}\left( \Omega \right) $ by the formula $\left( u\left( t\right)
,u_{t}\left( t\right) \right) =S\left( t\right) (u_{0},u_{1})$.

In this section, our aim is to study the existence of the global attractor
of the semigroup $\left\{ S\left( t\right) \right\} _{t\geq 0}$. The
attractors for the semilinear plate equations were studied by many authors
under different conditions. We refer to [13-18] and therein references. The
main difficulty in the proof of the existence of a global attractor for the
plate equations with the localized damping is to show the existence of a
bounded absorbing set, which is equivalent to point dissipativity for the
asymptotically compact semigroups (see [19]). The validity of Theorem 1.1
allows us to overcome this difficulty for (3.1) and we show that the
semigroup $\left\{ S\left( t\right) \right\} _{t\geq 0}$ generated by the
problem (3.1)-(3.3) has a global attractor which equals the unstable
manifold of the set of stationary points. To this end, at first, we prove
the asymptotic compactness of $\left\{ S\left( t\right) \right\} _{t\geq 0}$
in $H_{0}^{2}\left( \Omega \right) \times L^{2}\left( \Omega \right) $.

\begin{lemma}
Let conditions (3.4)-(3.6) hold and $B$ be a bounded subset of $%
H_{0}^{2}\left( \Omega \right) \times L^{2}\left( \Omega \right) $. Then
every sequence of the form $\left\{ S\left( t_{k}\right) \varphi
_{k}\right\} _{k=1}^{\infty }$, where $\left\{ \varphi _{k}\right\}
_{k=1}^{\infty }\subset B$, $t_{k}\rightarrow \infty $ , has a convergent
subsequence in $H_{0}^{2}\left( \Omega \right) \times L^{2}\left( \Omega
\right) $.
\end{lemma}

\begin{proof}
Since $\left\{ \varphi _{k}\right\} _{k=1}^{\infty }$ is bounded in $%
H_{0}^{2}\left( \Omega \right) \times L^{2}\left( \Omega \right) $, by using
(3.4) and (3.10), we obtain that $\left\{ S\left( \cdot \right) \varphi
_{k}\right\} _{k=1}^{\infty }$ is bounded in $L^{\infty }\left( 0,\infty
;H_{0}^{2}\left( \Omega \right) \times L^{2}\left( \Omega \right) \right) $.
Then, for any $T_{0}\geq 0$ there exists a subsequence $\left\{
k_{m}\right\} _{m=1}^{\infty }$ such that $t_{k_{m}}\geq T_{0}$ and%
\begin{equation}
\left\{
\begin{array}{c}
S\left( t_{k_{m}}-T_{0}\right) \varphi _{k_{m}}\rightarrow \varphi _{0}\text{
weakly in }H_{0}^{2}\left( \Omega \right) \times L^{2}\left( \Omega \right)
\text{,} \\
v_{m}\rightarrow v\text{ weakly star in }L^{\infty }\left( 0,\infty
;H_{0}^{2}\left( \Omega \right) \right) \text{,} \\
v_{mt}\rightarrow v_{t}\text{ weakly star in }L^{\infty }\left( 0,\infty
;L^{2}\left( \Omega \right) \right) \text{,} \\
v_{m}\left( t\right) \rightarrow v\left( t\right) \text{ weakly in }%
H_{0}^{2}\left( \Omega \right) \text{, }\forall t\geq 0\text{,}%
\end{array}%
\right.  \tag{3.11}
\end{equation}%
for some $\varphi _{0}\in H_{0}^{2}\left( \Omega \right) \times L^{2}\left(
\Omega \right) $ and $v\in L^{\infty }\left( 0,\infty ;H_{0}^{2}\left(
\Omega \right) \right) \cap W^{1,\infty }\left( 0,\infty ;L^{2}\left( \Omega
\right) \right) $, where $\left( v_{m}(t),v_{mt}(t)\right)
=S(t+t_{k_{m}}-T_{0})\varphi _{k_{m}}$. By using (3.5) and (3.10), we also
obtain%
\begin{equation}
\int\limits_{0}^{T}\left\Vert v_{mt}\left( t\right) \right\Vert
_{L^{2}\left( \omega \right) }^{2}dt\leq c_{1}\text{, \ \ \ }\forall T\geq 0%
\text{.\ }  \tag{3.12}
\end{equation}%
By using (3.1), we have the following equation
\begin{equation*}
v_{mtt}(t,x)-v_{ltt}(t,x)+\Delta ^{2}(v_{m}(t,x)-v_{l}(t,x))+\alpha
(x)(v_{mt}(t,x)-v_{lt}(t,x))
\end{equation*}%
\begin{equation*}
-f_{1}(\left\Vert \nabla v_{m}\left( t\right) \right\Vert _{L^{2}\left(
\Omega \right) })\Delta v_{m}(t,x)+f_{1}(\left\Vert \nabla v_{l}\left(
t\right) \right\Vert _{L^{2}\left( \Omega \right) })\Delta v_{l}(t,x)+\beta
(x_{1})(v_{m}(t,x)-v_{l}(t,x))
\end{equation*}%
\begin{equation}
+f_{2}(\left\Vert v_{m}\left( t\right) \right\Vert _{L^{2}\left( \Omega
\right) })v_{m}(t,x)-f_{2}(\left\Vert v_{l}\left( t\right) \right\Vert
_{L^{2}\left( \Omega \right) })v_{l}(t,x)=0\text{.}  \tag{3.13}
\end{equation}%
\newline
Now, let $\widetilde{\omega }\subset \Omega $ be also a neighbourhood of the
boundary $\partial \Omega $ such that $\widetilde{\omega }\subset \omega $
and $\Omega \backslash \omega \Subset \Omega \backslash \widetilde{\omega }$%
. Let $\eta \in C^{2}\left( \overline{\Omega }\right) $, $0\leq \eta \left(
x\right) \leq 1$, $\eta |_{\widetilde{\omega }}\equiv 1$ and $\eta |_{\Omega
\backslash \omega }\equiv 0$. Testing the equation (3.13) by $\eta
^{2}\left( v_{m}-v_{l}\right) $ on $\left( 0,T\right) \times \Omega $, using
integration by parts, and considering $\eta |_{\Omega \backslash \omega
}\equiv 0$ and $\beta (\cdot )\geq 0$, we get
\begin{equation*}
\int\limits_{0}^{T}\left\Vert \eta \Delta \left( v_{m}\left( t\right)
-v_{l}\left( t\right) \right) \right\Vert _{L^{2}\left( \omega \right)
}^{2}dt\leq \int\limits_{0}^{T}\left\Vert \eta \left( v_{mt}\left( t\right)
-v_{lt}\left( t\right) \right) \right\Vert _{L^{2}\left( \omega \right)
}^{2}dt
\end{equation*}%
\begin{equation*}
-\left. \left( \int\limits_{\omega }\left( v_{mt}\left( t,x\right)
-v_{lt}(\left( t,x\right) \right) \eta ^{2}\left( x\right) \left(
v_{m}\left( t,x\right) -v_{l}(\left( t,x\right) \right) dx\right)
\right\vert _{0}^{T}
\end{equation*}%
\begin{equation*}
-\int\limits_{0}^{T}\int\limits_{\omega }\Delta \left( v_{m}\left(
t,x\right) -v_{l}(\left( t,x\right) \right) \Delta \left( \eta ^{2}\left(
x\right) \right) \left( v_{m}\left( t,x\right) -v_{l}(\left( t,x\right)
\right) dxdt
\end{equation*}%
\begin{equation*}
-4\sum\limits_{i=1}^{n}\int\limits_{0}^{T}\int\limits_{\omega }\Delta \left(
v_{m}\left( t,x\right) -v_{l}(\left( t,x\right) \right) \eta \left( x\right)
\eta _{x_{i}}\left( x\right) \left( v_{m}\left( t,x\right) -v_{l}(\left(
t,x\right) \right) _{x_{i}}dxdt
\end{equation*}%
\begin{equation*}
-\int\limits_{0}^{T}\int\limits_{\omega }\alpha \left( x\right) \left(
v_{mt}\left( t,x\right) -v_{lt}\left( t,x\right) \right) \eta ^{2}\left(
x\right) \left( v_{m}\left( t,x\right) -v_{l}\left( t,x\right) \right) dxdt
\end{equation*}%
\begin{equation*}
+\int\limits_{0}^{T}\int\limits_{\omega }f_{1}\left( \left\Vert \nabla
v_{m}\left( t\right) \right\Vert _{L^{2}\left( \Omega \right) }\right)
\Delta v_{m}\left( t,x\right) \eta ^{2}\left( x\right) \left( v_{m}\left(
t,x\right) -v_{l}\left( t,x\right) \right) dxdt
\end{equation*}%
\begin{equation*}
-\int\limits_{0}^{T}\int\limits_{\omega }f_{1}\left( \left\Vert \nabla
v_{l}\left( t\right) \right\Vert _{L^{2}\left( \Omega \right) }\right)
\Delta v_{l}\left( t,x\right) \eta ^{2}\left( x\right) \left( v_{m}\left(
t,x\right) -v_{l}\left( t,x\right) \right) dxdt
\end{equation*}%
\begin{equation*}
+\int\limits_{0}^{T}\int\limits_{\omega }f_{2}\left( \left\Vert v_{m}\left(
t\right) \right\Vert _{L^{2}\left( \Omega \right) }\right) v_{m}\left(
t,x\right) \eta ^{2}\left( x\right) \left( v_{m}\left( t,x\right)
-v_{l}\left( t,x\right) \right) dxdt
\end{equation*}%
\begin{equation*}
-\int\limits_{0}^{T}\int\limits_{\omega }f_{2}\left( \left\Vert v_{l}\left(
t\right) \right\Vert _{L^{2}\left( \Omega \right) }\right) v_{l}\left(
t,x\right) \eta ^{2}\left( x\right) \left( v_{m}\left( t,x\right)
-v_{l}\left( t,x\right) \right) dxdt\text{.}
\end{equation*}%
Considering $\eta |_{\widetilde{\omega }}\equiv 1$ in the above inequality
and using (3.5), (3.7), (3.10) and (3.12), we have%
\begin{equation*}
\int\limits_{0}^{T}\left\Vert \Delta \left( v_{m}(t)-v_{l}(t)\right)
\right\Vert _{L^{2}\left( \widetilde{\omega }\right) }^{2}dt\leq
c_{2}+c_{3}T\left\Vert v_{m}-v_{l}\right\Vert _{C\left( \left[ 0,T\right]
;H_{0}^{1}\left( \Omega \right) \right) }\text{.}
\end{equation*}%
Since the sequence $\left\{ v_{m}\right\} _{m=1}^{\infty }$ is bounded in $%
L^{\infty }\left( 0,T;H_{0}^{2}\left( \Omega \right) \right) $ and the
sequence $\left\{ v_{mt}\right\} _{m=1}^{\infty }$ is bounded in $L^{\infty
}\left( 0,T;L^{2}\left( \Omega \right) \right) $, by the compact embedding
theorem (see [20, Corollary 4]), $\left\{ v_{m}\right\} _{m=1}^{\infty }$ $\
$is relatively compact in $C\left( \left[ 0,T\right] ;H_{0}^{1}\left( \Omega
\right) \right) $. So according to (3.11), the sequence $\left\{
v_{m}\right\} _{m=1}^{\infty }$ strongly converges to $v$ in $C\left( \left[
0,T\right] ;H_{0}^{1}\left( \Omega \right) \right) $. Then, by the last
inequality, we obtain,
\begin{equation}
\limsup\limits_{m\rightarrow \infty }\limsup\limits_{l\rightarrow \infty
}\int\limits_{0}^{T}\left\Vert v_{m}\left( t\right) -v_{l}\left( t\right)
\right\Vert _{H^{2}\left( \widetilde{\omega }\right) }^{2}dt\leq c_{2}\text{%
, \ \ }\forall T\geq 0\text{.}  \tag{3.14}
\end{equation}%
Now, let $q_{i}\in C_{0}^{2}\left( \Omega \right) $, $q_{i}|_{\Omega
\backslash \widetilde{\omega }}\equiv x_{i}$, $i=1,..,n$. Testing (3.13) by
\newline
$\sum\limits_{i=1}^{n}q_{i}\left( x\right) \left( v_{m}-v_{l}\right)
_{x_{i}}+\frac{1}{2}\left( n-1\right) \left( v_{m}-v_{l}\right) $ on $\left(
0,T\right) \times \Omega $ and using integration by parts, we get
\begin{equation*}
\sum\limits_{i=1}^{n}\left. \left( \int\limits_{\Omega }\left( v_{mt}\left(
t,x\right) -v_{lt}\left( t,x\right) \right) q_{i}\left( x\right) \left(
v_{m}\left( t,x\right) -v_{l}\left( t,x\right) \right) _{x_{i}}dx\right)
\right\vert _{0}^{T}
\end{equation*}%
\begin{equation*}
+\frac{1}{2}\sum\limits_{i=1}^{n}\int\limits_{0}^{T}\int\limits_{\Omega
}\left( q_{i}\left( x\right) \right) _{x_{i}}\left( v_{mt}\left( t,x\right)
-v_{lt}\left( t,x\right) \right) ^{2}dxdt
\end{equation*}%
\begin{equation*}
+\frac{1}{2}\left( n-1\right) \left. \left( \int\limits_{\Omega }\left(
v_{mt}\left( t,x\right) -v_{lt}\left( t,x\right) \right) \left( v_{m}\left(
t,x\right) -v_{l}\left( t,x\right) \right) dx\right) \right\vert _{0}^{T}
\end{equation*}%
\begin{equation*}
-\frac{1}{2}\left( n-1\right) \int\limits_{0}^{T}\left\Vert v_{mt}\left(
t\right) -v_{lt}\left( t\right) \right\Vert _{L^{2}\left( \Omega \right)
}^{2}dt
\end{equation*}%
\begin{equation*}
+\sum\limits_{i=1}^{n}\int\limits_{0}^{T}\int\limits_{\Omega }\left( \Delta
v_{m}\left( t,x\right) -\Delta v_{l}\left( t,x\right) \right) \Delta \left(
q_{i}\left( x\right) \right) \left( v_{m}\left( t,x\right) -v_{l}\left(
t,x\right) \right) _{x_{i}}dxdt
\end{equation*}%
\begin{equation*}
+2\sum\limits_{i,j=1}^{n}\int\limits_{0}^{T}\int\limits_{\Omega }\left(
\Delta v_{m}\left( t,x\right) -\Delta v_{l}\left( t,x\right) \right)
(q_{i}(x))_{x_{j}}\left( v_{m}\left( t,x\right) -v_{l}\left( t,x\right)
\right) _{x_{i}x_{j}}dxdt
\end{equation*}%
\begin{equation*}
-\frac{1}{2}\sum\limits_{i=1}^{n}\int\limits_{0}^{T}\int\limits_{\Omega
}(q_{i}\left( x\right) )_{x_{i}}\left( \Delta v_{m}\left( t,x\right) -\Delta
v_{l}\left( t,x\right) \right) ^{2}dxdt
\end{equation*}%
\begin{equation*}
+\frac{1}{2}\left( n-1\right) \int\limits_{0}^{T}\left\Vert \left( \Delta
v_{m}\left( t\right) -\Delta v_{l}\left( t\right) \right) \right\Vert
_{L^{2}\left( \Omega \right) }^{2}dt
\end{equation*}%
\begin{equation*}
+\sum\limits_{i=1}^{n}\int\limits_{0}^{T}\int\limits_{\Omega }\alpha \left(
x\right) \left( v_{mt}\left( t,x\right) -v_{lt}\left( t,x\right) \right)
q_{i}\left( x\right) \left( v_{m}\left( t,x\right) -v_{l}\left( t,x\right)
\right) _{x_{i}}dxdt
\end{equation*}%
\begin{equation*}
+\frac{1}{2}\left( n-1\right) \int\limits_{0}^{T}\int\limits_{\Omega }\alpha
\left( x\right) \left( v_{mt}\left( t,x\right) -v_{lt}\left( t,x\right)
\right) \left( v_{m}\left( t,x\right) -v_{l}\left( t,x\right) \right) dxdt
\end{equation*}%
\begin{equation*}
-\sum\limits_{i=1}^{n}\int\limits_{0}^{T}\int\limits_{\Omega }f_{1}\left(
\left\Vert \nabla v_{m}\left( t\right) \right\Vert _{L^{2}\left( \Omega
\right) }\right) \Delta v_{m}\left( t,x\right) q_{i}\left( x\right) \left(
v_{m}\left( t,x\right) -v_{l}\left( t,x\right) \right) _{x_{i}}dxdt
\end{equation*}%
\begin{equation*}
-\frac{1}{2}\left( n-1\right) \int\limits_{0}^{T}\int\limits_{\Omega
}f_{1}\left( \left\Vert \nabla v_{m}\left( t\right) \right\Vert
_{L^{2}\left( \Omega \right) }\right) \Delta v_{m}\left( t,x\right) \left(
v_{m}\left( t,x\right) -v_{l}\left( t,x\right) \right) dxdt
\end{equation*}%
\begin{equation*}
+\sum\limits_{i=1}^{n}\int\limits_{0}^{T}\int\limits_{\Omega }f_{1}\left(
\left\Vert \nabla v_{l}\left( t\right) \right\Vert _{L^{2}\left( \Omega
\right) }\right) \Delta v_{l}\left( t,x\right) q_{i}\left( x\right) \left(
v_{m}\left( t,x\right) -v_{l}\left( t,x\right) \right) _{x_{i}}dxdt
\end{equation*}%
\begin{equation*}
+\frac{1}{2}\left( n-1\right)
\sum\limits_{i=1}^{n}\int\limits_{0}^{T}\int\limits_{\Omega }f_{1}\left(
\left\Vert \nabla v_{l}\left( t\right) \right\Vert _{L^{2}\left( \Omega
\right) }\right) \Delta v_{l}\left( t,x\right) \left( v_{m}\left( t,x\right)
-v_{l}\left( t,x\right) \right) dxdt
\end{equation*}%
\begin{equation*}
+\sum\limits_{i=1}^{n}\int\limits_{0}^{T}\int\limits_{\Omega }f_{2}\left(
\left\Vert v_{m}\left( t\right) \right\Vert _{L^{2}\left( \Omega \right)
}\right) q_{i}\left( x\right) v_{m}\left( t,x\right) \left( v_{m}\left(
t,x\right) -v_{l}\left( t,x\right) \right) _{x_{i}}dxdt
\end{equation*}%
\begin{equation*}
+\frac{1}{2}\left( n-1\right) \int\limits_{0}^{T}\int\limits_{\Omega
}f_{2}\left( \left\Vert v_{m}\left( t\right) \right\Vert _{L^{2}\left(
\Omega \right) }\right) v_{m}\left( t,x\right) \left( v_{m}\left( t,x\right)
-v_{l}\left( t,x\right) \right) dxdt
\end{equation*}%
\begin{equation*}
-\sum\limits_{i=1}^{n}\int\limits_{0}^{T}\int\limits_{\Omega }f_{2}\left(
\left\Vert v_{l}\left( t\right) \right\Vert _{L^{2}\left( \Omega \right)
}\right) q_{i}\left( x\right) v_{l}\left( t,x\right) \left( v_{m}\left(
t,x\right) -v_{l}\left( t,x\right) \right) _{x_{i}}dxdt
\end{equation*}%
\begin{equation*}
-\frac{1}{2}\left( n-1\right) \int\limits_{0}^{T}\int\limits_{\Omega
}f_{2}\left( \left\Vert v_{l}\left( t\right) \right\Vert _{L^{2}\left(
\Omega \right) }\right) v_{l}\left( t,x\right) \left( v_{m}\left( t,x\right)
-v_{l}\left( t,x\right) \right) dxdt
\end{equation*}%
\begin{equation*}
+\sum\limits_{i=1}^{n}\int\limits_{0}^{T}\int\limits_{\Omega }\beta
(x_{1})(v_{m}(t,x)-v_{l}(t,x))q_{i}\left( x\right) \left( v_{m}\left(
t,x\right) -v_{l}\left( t,x\right) \right) _{x_{i}}dxdt\leq 0.
\end{equation*}%
Then, since $q_{i}|_{\Omega \backslash \widetilde{\omega }}\equiv x_{i}$,
from the last inequality, we find,%
\begin{equation*}
\frac{3}{2}\int\limits_{0}^{T}\left\Vert \left( \Delta v_{m}\left( t\right)
-\Delta v_{l}\left( t\right) \right) \right\Vert _{L^{2}\left( \Omega
\backslash \widetilde{\omega }\right) }^{2}dt+\frac{1}{2}\int\limits_{0}^{T}%
\left\Vert v_{mt}\left( t\right) -v_{lt}\left( t\right) \right\Vert
_{L^{2}\left( \Omega \backslash \widetilde{\omega }\right) }^{2}dt
\end{equation*}%
\begin{equation*}
\leq -\left. \left( \sum\limits_{i=1}^{n}\int\limits_{\Omega }\left(
v_{mt}\left( t,x\right) -v_{lt}\left( t,x\right) \right) q_{i}\left(
x\right) \left( v_{m}\left( t,x\right) -v_{l}\left( t,x\right) \right)
_{x_{i}}dx\right) \right\vert _{0}^{T}
\end{equation*}%
\begin{equation*}
-\frac{1}{2}\sum\limits_{i=1}^{n}\int\limits_{0}^{T}\int\limits_{\widetilde{%
\omega }}\left( q_{i}\left( x\right) \right) _{x_{i}}\left( v_{mt}\left(
t,x\right) -v_{lt}\left( t,x\right) \right) ^{2}dxdt
\end{equation*}%
\begin{equation*}
-\frac{1}{2}(n-1)\left. \left( \int\limits_{\Omega }(\left( v_{mt}\left(
t\right) -v_{lt}\left( t\right) \right) \left( v_{m}\left( t,x\right)
-v_{l}\left( t,x\right) \right) dx\right) \right\vert _{0}^{T}
\end{equation*}%
\begin{equation*}
+\frac{1}{2}(n-1)\int\limits_{0}^{T}\left\Vert v_{mt}\left( t\right)
-v_{lt}\left( t\right) \right\Vert _{L^{2}\left( \widetilde{\omega }\right)
}^{2}dt
\end{equation*}%
\begin{equation*}
-\sum\limits_{i=1}^{n}\int\limits_{0}^{T}\int\limits_{\Omega }\Delta \left(
v_{m}\left( t,x\right) -v_{l}\left( t,x\right) \right) \Delta \left(
q_{i}\left( x\right) \right) \left( v_{m}\left( t,x\right) -v_{l}\left(
t,x\right) \right) _{x_{i}}dxdt
\end{equation*}%
\begin{equation*}
-2\sum\limits_{i,j=1}^{n}\int\limits_{0}^{T}\int\limits_{\widetilde{\omega }%
}\Delta (v_{m}(t,x)-v_{l}(t,x))(q_{i}(x))_{x_{j}}\left( v_{m}\left(
t,x\right) -v_{l}\left( t,x\right) \right) _{x_{i}x_{j}}dxdt
\end{equation*}%
\begin{equation*}
+\frac{1}{2}\sum\limits_{i=1}^{n}\int\limits_{0}^{T}\int\limits_{\widetilde{%
\omega }}\left( q_{i}\left( x\right) \right) _{x_{i}}\left( \Delta \left(
v_{m}\left( t,x\right) -v_{l}\left( t,x\right) \right) \right) ^{2}dxdt
\end{equation*}%
\begin{equation*}
-\frac{1}{2}\left( n-1\right) \int\limits_{0}^{T}\left\Vert \Delta \left(
v_{m}\left( t\right) -v_{l}\left( t\right) \right) \right\Vert _{L^{2}\left(
\widetilde{\omega }\right) }^{2}dt
\end{equation*}%
\begin{equation*}
-\sum\limits_{i=1}^{n}\int\limits_{0}^{T}\int\limits_{\Omega }\alpha \left(
x\right) \left( v_{mt}\left( t,x\right) -v_{lt}\left( t,x\right) \right)
q_{i}\left( x\right) \left( v_{m}\left( t,x\right) -v_{l}\left( t,x\right)
\right) _{x_{i}}dxdt
\end{equation*}%
\begin{equation*}
-\frac{1}{2}(n-1)\int\limits_{0}^{T}\int\limits_{\Omega }\alpha \left(
x\right) (v_{mt}\left( t,x\right) -v_{lt}\left( t,x\right) )(v_{m}\left(
t,x\right) -v_{l}\left( t,x\right) )dxdt
\end{equation*}%
\begin{equation*}
+\sum\limits_{i=1}^{n}\int\limits_{0}^{T}\int\limits_{\Omega }f_{1}\left(
\left\Vert \nabla v_{m}\left( t\right) \right\Vert _{L^{2}\left( \Omega
\right) }\right) \Delta v_{m}\left( t,x\right) q_{i}\left( x\right) \left(
v_{m}\left( t,x\right) -v_{l}\left( t,x\right) \right) _{x_{i}}dxdt
\end{equation*}%
\begin{equation*}
+\frac{1}{2}(n-1)\int\limits_{0}^{T}\int\limits_{\Omega }f_{1}\left(
\left\Vert \nabla v_{m}\left( t\right) \right\Vert _{L^{2}\left( \Omega
\right) }\right) \Delta v_{m}\left( t,x\right) \left( v_{m}\left( t,x\right)
-v_{l}\left( t,x\right) \right) dxdt
\end{equation*}%
\begin{equation*}
-\sum\limits_{i=1}^{n}\int\limits_{0}^{T}\int\limits_{\Omega }f_{1}\left(
\left\Vert \nabla v_{l}\left( t\right) \right\Vert _{L^{2}\left( \Omega
\right) }\right) \Delta v_{l}\left( t,x\right) q_{i}\left( x\right) \left(
v_{m}\left( t,x\right) -v_{l}\left( t,x\right) \right) _{x_{i}}dxdt
\end{equation*}%
\begin{equation*}
-\frac{1}{2}(n-1)\int\limits_{0}^{T}\int\limits_{\Omega }f_{1}\left(
\left\Vert \nabla v_{l}\left( t\right) \right\Vert _{L^{2}\left( \Omega
\right) }\right) \Delta v_{l}\left( t,x\right) \left( v_{m}\left( t,x\right)
-v_{l}\left( t,x\right) \right) dxdt
\end{equation*}%
\begin{equation*}
-\sum\limits_{i=1}^{n}\int\limits_{0}^{T}\int\limits_{\Omega }f_{2}\left(
\left\Vert v_{m}\left( t\right) \right\Vert _{L^{2}\left( \Omega \right)
}\right) v_{m}\left( t,x\right) q_{i}\left( x\right) \left( v_{m}\left(
t,x\right) -v_{l}\left( t,x\right) \right) _{x_{i}}dxdt
\end{equation*}%
\begin{equation*}
-\frac{1}{2}(n-1)\int\limits_{0}^{T}\int\limits_{\Omega }f_{2}\left(
\left\Vert v_{m}\left( t\right) \right\Vert _{L^{2}\left( \Omega \right)
}\right) v_{m}\left( t,x\right) \left( v_{m}\left( t,x\right) -v_{l}\left(
t,x\right) \right) dxdt
\end{equation*}%
\begin{equation*}
+\sum\limits_{i=1}^{n}\int\limits_{0}^{T}\int\limits_{\Omega }f_{2}\left(
\left\Vert v_{l}\left( t\right) \right\Vert _{L^{2}\left( \Omega \right)
}\right) \left( v_{l}\left( t,x\right) \right) q_{i}\left( x\right) \left(
v_{m}\left( t,x\right) -v_{l}\left( t,x\right) \right) _{x_{i}}dxdt
\end{equation*}%
\begin{equation*}
+\frac{1}{2}\left( n-1\right) \int\limits_{0}^{T}\int\limits_{\Omega
}f_{2}\left( \left\Vert v_{l}\left( t\right) \right\Vert _{L^{2}\left(
\Omega \right) }\right) v_{l}\left( t,x\right) \left( v_{m}\left( t,x\right)
-v_{l}\left( t,x\right) \right) dxdt
\end{equation*}%
\begin{equation*}
-\sum\limits_{i=1}^{n}\int\limits_{0}^{T}\int\limits_{\Omega }\beta
(x_{1})(v_{m}(t,x)-v_{l}(t,x))q_{i}\left( x\right) \left( v_{m}\left(
t,x\right) -v_{l}\left( t,x\right) \right) _{x_{i}}dxdt\text{.}
\end{equation*}%
Taking into account (3.4), (3.6), (3.7), (3.10) and (3.12) in the last
inequality, we get%
\begin{equation*}
\int\limits_{0}^{T}\left\Vert \Delta v_{m}\left( t\right) -\Delta
v_{l}\left( t\right) \right\Vert _{L^{2}\left( \Omega \backslash \widetilde{%
\omega }\right) }^{2}dt+\int\limits_{0}^{T}\left\Vert v_{mt}\left( t\right)
-v_{lt}\left( t\right) \right\Vert _{L^{2}\left( \Omega \backslash
\widetilde{\omega }\right) }^{2}dt
\end{equation*}%
\begin{equation*}
\leq c_{4}+c_{5}T\left\Vert v_{m}-v_{l}\right\Vert _{C\left( \left[ 0,T%
\right] ;H_{0}^{1}\left( \Omega \right) \right)
}+c_{6}\int\limits_{0}^{T}\left\Vert v_{m}\left( t\right) -v_{l}\left(
t\right) \right\Vert _{H^{2}\left( \widetilde{\omega }\right) }^{2}dt\text{,}
\end{equation*}%
which, together with (3.14) and strong convergence of $\left\{ v_{m}\right\}
_{m=1}^{\infty }$ in $C\left( \left[ 0,T\right] ;H_{0}^{1}\left( \Omega
\right) \right) $, yields%
\begin{equation*}
\limsup\limits_{m\rightarrow \infty }\limsup\limits_{l\rightarrow \infty
}\int\limits_{0}^{T}\left[ \left\Vert \Delta v_{m}\left( t\right) -\Delta
v_{l}\left( t\right) \right\Vert _{L^{2}(\Omega \backslash \widetilde{\omega
})}^{2}+\left\Vert v_{mt}\left( t\right) -v_{lt}\left( t\right) \right\Vert
_{L^{2}(\Omega \backslash \widetilde{\omega })}^{2}\right] dt
\end{equation*}%
\begin{equation}
\leq c_{7}\text{, \ \ \ \ }\forall T\geq 0\text{.}  \tag{3.15}
\end{equation}%
By (3.12), (3.14) and (3.15), we obtain%
\begin{equation*}
\limsup\limits_{m\rightarrow \infty }\limsup\limits_{l\rightarrow \infty
}\int\limits_{0}^{T}\left[ \left\Vert v_{m}\left( t\right) -v_{l}\left(
t\right) \right\Vert _{H_{0}^{2}\left( \Omega \right) }^{2}+\left\Vert
v_{mt}\left( t\right) -v_{lt}\left( t\right) \right\Vert _{L^{2}\left(
\Omega \right) }^{2}\right] dt
\end{equation*}%
\begin{equation}
\leq c_{8}\text{, \ \ \ \ }\forall T\geq 0\text{.}  \tag{3.16}
\end{equation}%
Now, multiplying (3.13) by $2t\left( v_{mt}-v_{lt}\right) $, integrating
over $\left( 0,T\right) \times \Omega $ and using integration by parts, we
have%
\begin{equation*}
T\left\Vert v_{mt}\left( T\right) -v_{lt}\left( T\right) \right\Vert
_{L^{2}\left( \Omega \right) }^{2}
\end{equation*}%
\begin{equation*}
-\int\limits_{0}^{T}\left\Vert v_{mt}\left( t\right) -v_{lt}\left( t\right)
\right\Vert _{L^{2}\left( \Omega \right) }^{2}dt+T\left\Vert \Delta
v_{m}(T)-\Delta v_{l}(T)\right\Vert _{L\left( \Omega \right)
}^{2}-\int\limits_{0}^{T}\left\Vert \Delta v_{m}(t)-\Delta
v_{l}(t)\right\Vert _{L^{2}\left( \Omega \right) }^{2}dt
\end{equation*}%
\begin{equation*}
+T\left( F_{1}\left( \left\Vert \nabla v_{m}\left( T\right) \right\Vert
_{L^{2}\left( \Omega \right) }^{2}\right) -\frac{1}{T}\int%
\limits_{0}^{T}F_{1}\left( \left\Vert \nabla v_{m}\left( t\right)
\right\Vert _{L^{2}\left( \Omega \right) }^{2}\right) dt\right)
\end{equation*}%
\begin{equation*}
+2\int\limits_{0}^{T}\int\limits_{\Omega }tf_{1}\left( \left\Vert \nabla
v_{m}\left( t\right) \right\Vert _{L^{2}\left( \Omega \right) }\right)
\Delta v_{m}\left( t,x\right) v_{lt}\left( t,x\right) dxdt
\end{equation*}%
\begin{equation*}
+2\int\limits_{0}^{T}\int\limits_{\Omega }tf_{1}\left( \left\Vert \nabla
v_{l}\left( t\right) \right\Vert _{L^{2}\left( \Omega \right) }\right)
\Delta v_{l}\left( t,x\right) v_{mt}\left( t,x\right) dxdt
\end{equation*}%
\begin{equation*}
+T\left( F_{1}\left( \left\Vert \nabla v_{l}\left( T\right) \right\Vert
_{L^{2}\left( \Omega \right) }^{2}\right) -\frac{1}{T}\int%
\limits_{0}^{T}F_{1}\left( \left\Vert \nabla v_{l}\left( t\right)
\right\Vert _{L^{2}\left( \Omega \right) }^{2}\right) dt\right)
\end{equation*}%
\begin{equation*}
+T\left( F_{2}\left( \left\Vert v_{m}\left( T\right) \right\Vert
_{L^{2}\left( \Omega \right) }^{2}\right) -\frac{1}{T}\int%
\limits_{0}^{T}F_{2}\left( \left\Vert v_{m}\left( t\right) \right\Vert
_{L^{2}\left( \Omega \right) }^{2}\right) dt\right)
\end{equation*}%
\begin{equation*}
-2\int\limits_{0}^{T}\int\limits_{\Omega }tf_{2}\left( \left\Vert
v_{m}\left( t\right) \right\Vert _{L^{2}\left( \Omega \right) }\right)
v_{m}\left( t,x\right) v_{lt}\left( t,x\right) dxdt
\end{equation*}%
\begin{equation*}
-2\int\limits_{0}^{T}\int\limits_{\Omega }tf_{2}\left( \left\Vert
v_{l}\left( t\right) \right\Vert _{L^{2}\left( \Omega \right) }\right)
v_{l}\left( t,x\right) v_{mt}\left( t,x\right) dxdt
\end{equation*}%
\begin{equation*}
+T\left( F_{2}\left( \left\Vert v_{l}\left( T\right) \right\Vert
_{L^{2}\left( \Omega \right) }^{2}\right) -\frac{1}{T}\int%
\limits_{0}^{T}F_{2}\left( \left\Vert v_{l}\left( t\right) \right\Vert
_{L^{2}\left( \Omega \right) }^{2}\right) dt\right)
\end{equation*}%
\begin{equation*}
-\int\limits_{0}^{T}\int\limits_{\Omega }\beta (x_{1})\left\vert
v_{m}(t,x)-v_{l}(t,x)\right\vert ^{2}dxdt=0\text{.}
\end{equation*}%
Hence, dividing both sides of the above inequality by $T$, we obtain%
\begin{equation*}
\left\Vert \left( \Delta v_{m}\left( T\right) -\Delta v_{l}\left( T\right)
\right) \right\Vert _{L^{2}\left( \Omega \right) }^{2}+\left\Vert
v_{mt}\left( T\right) -v_{lt}\left( T\right) \right\Vert _{L^{2}\left(
\Omega \right) }^{2}
\end{equation*}%
\begin{equation*}
\leq \frac{1}{T}\left( \int\limits_{0}^{T}\left\Vert v_{mt}\left( t\right)
-v_{lt}\left( t\right) \right\Vert _{L^{2}\left( \Omega \right)
}^{2}+\int\limits_{0}^{T}\left\Vert \Delta v_{m}\left( t\right) -\Delta
v_{l}\left( t\right) \right\Vert _{L^{2}\left( \Omega \right) }^{2}dt\right)
\end{equation*}%
\begin{equation*}
-F_{1}\left( \left\Vert \nabla v_{m}\left( T\right) \right\Vert
_{L^{2}\left( \Omega \right) }^{2}\right) +\frac{1}{T}\int%
\limits_{0}^{T}F_{1}\left( \left\Vert \nabla v_{m}\left( t\right)
\right\Vert _{L^{2}\left( \Omega \right) }^{2}\right) dt
\end{equation*}%
\begin{equation*}
-\frac{2}{T}\int\limits_{0}^{T}\int\limits_{\Omega }tf_{1}\left( \left\Vert
\nabla v_{m}\left( t\right) \right\Vert _{L^{2}\left( \Omega \right)
}\right) \Delta v_{m}\left( t,x\right) v_{lt}\left( t,x\right) dxdt
\end{equation*}%
\begin{equation*}
-\frac{2}{T}\int\limits_{0}^{T}\int\limits_{\Omega }tf_{1}\left( \left\Vert
\nabla v_{l}\left( t\right) \right\Vert _{L^{2}\left( \Omega \right)
}\right) \Delta v_{l}\left( t,x\right) v_{mt}\left( t,x\right) dxdt
\end{equation*}%
\begin{equation*}
-F_{1}\left( \left\Vert \nabla v_{l}\left( T\right) \right\Vert
_{L^{2}\left( \Omega \right) }^{2}\right) +\frac{1}{T}\int%
\limits_{0}^{T}F_{1}\left( \left\Vert \nabla v_{l}\left( t\right)
\right\Vert _{L^{2}\left( \Omega \right) }^{2}\right) dt
\end{equation*}%
\begin{equation*}
-F_{2}\left( \left\Vert v_{m}\left( T\right) \right\Vert _{L^{2}\left(
\Omega \right) }^{2}\right) +\frac{1}{T}\int\limits_{0}^{T}F_{2}\left(
\left\Vert v_{m}\left( t\right) \right\Vert _{L^{2}\left( \Omega \right)
}^{2}\right) dt
\end{equation*}%
\begin{equation*}
+\frac{2}{T}\int\limits_{0}^{T}\int\limits_{\Omega }tf_{2}\left( \left\Vert
v_{m}\left( t\right) \right\Vert _{L^{2}\left( \Omega \right) }\right)
v_{m}\left( t,x\right) v_{lt}\left( t,x\right) dxdt
\end{equation*}%
\begin{equation*}
+\frac{2}{T}\int\limits_{0}^{T}\int\limits_{\Omega }tf_{2}\left( \left\Vert
v_{l}\left( t\right) \right\Vert _{L^{2}\left( \Omega \right) }\right)
v_{l}\left( t,x\right) v_{mt}\left( t,x\right) dxdt-F_{2}\left( \left\Vert
v_{l}\left( T\right) \right\Vert _{L^{2}\left( \Omega \right) }^{2}\right)
\end{equation*}%
\begin{equation*}
+\frac{1}{T}\int\limits_{0}^{T}F_{2}\left( \left\Vert v_{l}\left( t\right)
\right\Vert _{L^{2}\left( \Omega \right) }^{2}\right) dt+\frac{1}{T}%
\int\limits_{0}^{T}\int\limits_{\Omega }\beta (x_{1})\left\vert
v_{m}(t,x)-v_{l}(t,x)\right\vert ^{2}dxdt\text{.}
\end{equation*}%
Since $\left\{ v_{m}\right\} _{m=1}^{\infty }$ strongly converges to $v$ in $%
C\left( \left[ 0,T\right] ;H_{0}^{1}\left( \Omega \right) \right) $ and $%
F_{1}$, $F_{2}\in C^{1}\left(
\mathbb{R}
^{+}\right) $, by considering (3.16) in the last inequality, we get
\begin{equation*}
\limsup\limits_{m\rightarrow \infty }\limsup\limits_{l\rightarrow \infty }
\left[ \left\Vert v_{m}\left( T\right) -v_{l}\left( T\right) \right\Vert
_{H_{0}^{2}\left( \Omega \right) }^{2}+\left\Vert v_{m}\left( T\right)
-v_{l}\left( T\right) \right\Vert _{L^{2}\left( \Omega \right) }^{2}\right]
\text{ }
\end{equation*}%
\begin{equation*}
\leq \frac{c_{9}}{T}-2F_{1}\left( \left\Vert \nabla v\left( T\right)
\right\Vert _{L^{2}\left( \Omega \right) }^{2}\right) +\frac{2}{T}%
\int\limits_{0}^{T}F_{1}\left( \left\Vert \nabla v\left( t\right)
\right\Vert _{L^{2}\left( \Omega \right) }^{2}\right) dt
\end{equation*}%
\begin{equation*}
-\frac{4}{T}\int\limits_{0}^{T}\int\limits_{\Omega }tf_{1}\left( \left\Vert
\nabla v\left( t\right) \right\Vert _{L^{2}\left( \Omega \right) }\right)
\Delta v\left( t,x\right) v_{t}\left( t,x\right) dxdt
\end{equation*}%
\begin{equation*}
-2F_{2}\left( \left\Vert v\left( T\right) \right\Vert _{L^{2}\left( \Omega
\right) }^{2}\right) +\frac{2}{T}\int\limits_{0}^{T}F_{2}\left( \left\Vert
v\left( t\right) \right\Vert _{L^{2}\left( \Omega \right) }^{2}\right) dt
\end{equation*}%
\begin{equation*}
+\frac{4}{T}\int\limits_{0}^{T}\int\limits_{\Omega }tf_{2}\left( \left\Vert
v\left( t\right) \right\Vert _{L^{2}\left( \Omega \right) }\right) v\left(
t,x\right) v_{t}\left( t,x\right) dxdt\text{,}
\end{equation*}%
which leads to%
\begin{equation*}
\limsup\limits_{m\rightarrow \infty }\limsup\limits_{l\rightarrow \infty }
\left[ \left\Vert v_{m}\left( T\right) -v_{l}\left( T\right) \right\Vert
_{H_{0}^{2}\left( \Omega \right) }^{2}+\left\Vert v_{m}\left( T\right)
-v_{l}\left( T\right) \right\Vert _{L^{2}\left( \Omega \right) }^{2}\right]
\text{ }
\end{equation*}%
\begin{equation*}
\leq \frac{c_{9}}{T}-2F_{1}\left( \left\Vert \nabla v\left( T\right)
\right\Vert _{L^{2}\left( \Omega \right) }^{2}\right) +\frac{2}{T}%
\int\limits_{0}^{T}F_{1}\left( \left\Vert \nabla v\left( t\right)
\right\Vert _{L^{2}\left( \Omega \right) }^{2}\right) dt
\end{equation*}%
\begin{equation*}
+\frac{2}{T}\int\limits_{0}^{T}t\frac{d}{dt}F_{1}\left( \left\Vert \nabla
v\left( t\right) \right\Vert _{L^{2}\left( \Omega \right) }^{2}\right) dt
\end{equation*}%
\begin{equation*}
-2F_{2}\left( \left\Vert v\left( T\right) \right\Vert _{L^{2}\left( \Omega
\right) }^{2}\right) +\frac{2}{T}\int\limits_{0}^{T}F_{2}\left( \left\Vert
v\left( t\right) \right\Vert _{L^{2}\left( \Omega \right) }^{2}\right) dt
\end{equation*}%
\begin{equation*}
+\frac{2}{T}\int\limits_{0}^{T}t\frac{d}{dt}F_{2}\left( \left\Vert v\left(
t\right) \right\Vert _{L^{2}\left( \Omega \right) }^{2}\right) dt\text{.}
\end{equation*}%
Then, after integration by parts, we obtain%
\begin{equation*}
\limsup\limits_{m\rightarrow \infty }\limsup\limits_{l\rightarrow \infty }
\left[ \left\Vert v_{m}\left( T\right) -v_{l}\left( T\right) \right\Vert
_{H_{0}^{2}\left( \Omega \right) }^{2}+\left\Vert v_{m}\left( T\right)
-v_{l}\left( T\right) \right\Vert _{L^{2}\left( \Omega \right) }^{2}\right]
\text{ }
\end{equation*}%
\begin{equation*}
\leq \frac{c_{9}}{T}\text{, \ \ \ \ }\forall T>0\text{,}
\end{equation*}%
which yields%
\begin{equation*}
\limsup\limits_{m\rightarrow \infty }\limsup\limits_{l\rightarrow \infty
}\left\Vert S\left( T+t_{k_{m}}-T_{0}\right) \varphi _{k_{m}}-S\left(
T+t_{k_{l}}-T_{0}\right) \varphi _{k_{l}}\right\Vert _{H_{0}^{2}\left(
\Omega \right) \times L^{2}\left( \Omega \right) }
\end{equation*}%
\begin{equation*}
\leq \frac{c_{10}}{\sqrt{T}}\text{, \ \ \ \ \ }\forall T>0\text{.}
\end{equation*}%
Choosing $T=T_{0}$ in the above inequality, we have%
\begin{equation*}
\limsup\limits_{m\rightarrow \infty }\limsup\limits_{l\rightarrow \infty
}\left\Vert S\left( t_{k_{m}}\right) \varphi _{k_{m}}-S\left(
t_{k_{l}}\right) \varphi _{k_{l}}\right\Vert _{H_{0}^{2}\left( \Omega
\right) \times L^{2}\left( \Omega \right) }^{2}\leq \frac{c_{10}}{\sqrt{T_{0}%
}}\text{ , \ }\forall T_{0}\geq 1\text{,}
\end{equation*}%
and consequently, we get%
\begin{equation*}
\liminf\limits_{k\rightarrow \infty }\liminf\limits_{m\rightarrow \infty
}\left\Vert S\left( t_{k}\right) \varphi _{k}-S\left( t_{m}\right) \varphi
_{m}\right\Vert _{H_{0}^{2}\left( \Omega \right) \times L^{2}\left( \Omega
\right) }=0\text{.}
\end{equation*}%
Thus, by using the argument at the end of the proof of [21, Lemma 3.4], we
complete the proof of the lemma.
\end{proof}

Secondly, let us show the following point dissipativity property of $\left\{
S\left( t\right) \right\} _{t\geq 0}$.

\begin{lemma}
Under the conditions (3.4)-(3.7),
\begin{equation*}
\lim_{t\rightarrow \infty }\inf_{\phi \in \mathcal{N}}\left\Vert S\left(
t\right) \theta -\phi \right\Vert _{H_{0}^{2}\left( \Omega \right) \times
L^{2}\left( \Omega \right) }=0
\end{equation*}%
holds for every $\theta \in H_{0}^{2}\left( \Omega \right) \times
L^{2}\left( \Omega \right) $. Here $\mathcal{N}$ is the set of stationary
points (for definition, see [22, p.35]) of $\left\{ S\left( t\right)
\right\} _{t\geq 0}$.
\end{lemma}

\begin{proof}
Let $\theta \in H_{0}^{2}\left( \Omega \right) \times L^{2}\left( \Omega
\right) $. From asymptotically compactness lemma, it follows that the $%
\omega $-limit set of $\theta $, namely,%
\begin{equation*}
\omega \left( \theta \right) =\underset{t\geq 0}{\cap }\overline{\underset{%
\tau \geq t}{\cup }S\left( \tau \right) \theta }
\end{equation*}%
is nonempty, compact in $H_{0}^{2}\left( \Omega \right) \times L^{2}\left(
\Omega \right) $, invariant with respect to $S\left( t\right) $ and%
\begin{equation}
\underset{t\rightarrow \infty }{\lim }\underset{\phi \in \omega \left(
\theta \right) }{\inf }\left\Vert S\left( t\right) \theta -\phi \right\Vert
_{H_{0}^{2}\left( \Omega \right) \times L^{2}\left( \Omega \right) }=0\text{.%
}  \tag{3.17}
\end{equation}%
Let $\left( u\left( t\right) ,u_{t}\left( t\right) \right) =S\left( t\right)
\theta $. Since, by (3.4), (3.6), (3.7) and (3.9), the Lyapunov functional $%
L\left( S\left( t\right) \theta \right) :=E\left( u\left( t\right)
,u_{t}\left( t\right) \right) $ is nonincreasing and bounded from below, it
has a limit at positive infinity, i.e.
\begin{equation}
\underset{t\rightarrow \infty }{\lim }L(S\left( t\right) \theta )=l\text{.}
\tag{3.18}
\end{equation}%
If $\varphi \in \omega \left( \theta \right) $, there exists a sequence $%
\left\{ t_{m}\right\} _{m=1}^{\infty }\subset \left[ 0,\infty \right) $ such
that $t_{m}\rightarrow \infty $ and%
\begin{equation*}
S\left( t_{m}\right) \theta \rightarrow \varphi \text{ \ strongly in \ }%
H_{0}^{2}\left( \Omega \right) \times L^{2}\left( \Omega \right) \text{ \ as
\ }t_{m}\rightarrow \infty \text{.}
\end{equation*}%
Since the Lyapunov functional $L$ is continuous on $H_{0}^{2}\left( \Omega
\right) \times L^{2}\left( \Omega \right) $, we get
\begin{equation*}
L\left( S\left( t_{m}\right) \theta \right) \rightarrow L\left( \varphi
\right) \text{ \ \ as\ \ }t_{m}\rightarrow \infty \text{,}
\end{equation*}%
which, together with (3.18), gives%
\begin{equation}
L\left( \varphi \right) =l\text{ , \ }\forall \varphi \in \omega \left(
\theta \right) \text{.}  \tag{3.19}
\end{equation}%
Since the damping term in (3.1) is linear, the semigroup $\left\{ S\left(
t\right) \right\} _{t\geq 0}$ can be extended to group $\left\{ S\left(
t\right) \right\} _{t\in
\mathbb{R}
}$ and considering invariance of $\omega \left( \theta \right) $, by (3.19),
we have
\begin{equation*}
L\left( S\left( t\right) \varphi \right) =l\text{, \ }\forall \varphi \in
\omega \left( \theta \right) \text{, \ }\forall t\in
\mathbb{R}
\text{.}
\end{equation*}%
Now, let$\ \varphi \in \omega \left( \theta \right) $ and set $\left(
v\left( t\right) ,v_{t}\left( t\right) \right) =S(t)\varphi $ for $t\in
\mathbb{R}
$. Then, using (3.9), which can be extended for all $t\geq s$ and $s\in
\mathbb{R}
,$ and considering the last equality, we obtain
\begin{equation*}
\int\nolimits_{s}^{t}\int\nolimits_{\Omega }\alpha \left( x\right)
\left\vert v_{t}\left( t,x\right) \right\vert ^{2}dxdt=0\text{, \ }\forall
t\geq s\text{, \ }s\in
\mathbb{R}
\text{.}
\end{equation*}%
Taking into account (3.4) in above equality, we get
\begin{equation*}
\alpha \left( x\right) \left\vert v_{t}\left( t,x\right) \right\vert ^{2}=0%
\text{, \ \ a.e. in }%
\mathbb{R}
\times \Omega \text{,}
\end{equation*}%
which, together with (3.1) and (3.5), gives us the following problem
\begin{equation}
\left\{
\begin{array}{l}
v_{tt}\left( t,x\right) +\Delta ^{2}v\left( t,x\right) -f_{1}\left(
\left\Vert \nabla v\left( t\right) \right\Vert _{L^{2}\left( \Omega \right)
}\right) \Delta v\left( t,x\right) \\
+f_{2}\left( \left\Vert v\left( t\right) \right\Vert _{L^{2}\left( \Omega
\right) }\right) v\left( t,x\right) +\beta (x_{1})v\left( t,x\right) =0\text{%
,\ \ \ }\left( t,x\right) \in
\mathbb{R}
\times \Omega \text{,} \\
v_{t}\left( t,x\right) =0\text{, \ \ }(t,x)\in
\mathbb{R}
\times \omega \text{.}%
\end{array}%
\right.  \tag{3.20}
\end{equation}%
Now, regarding (3.20), we will prove that%
\begin{equation}
v_{t}\left( t,x\right) =0\text{, \ \ a.e. in }\Omega \text{,}  \tag{3.21}
\end{equation}%
for all $t\in
\mathbb{R}
$. Firstly, assume that the term $f_{1}\left( \left\Vert \nabla v\left(
t\right) \right\Vert _{L^{2}\left( \Omega \right) }\right) $ is not
constant. Then\ there exist $t_{1}$, $t_{2}\in
\mathbb{R}
$ such that $f_{1}\left( \left\Vert \nabla v\left( t_{1}\right) \right\Vert
_{L^{2}\left( \Omega \right) }\right) \neq $ $f_{1}\left( \left\Vert \nabla
v\left( t_{2}\right) \right\Vert _{L^{2}\left( \Omega \right) }\right) $.
Since
\begin{equation*}
v_{t}\left( t,x\right) =0\text{, \ \ a.e. in }%
\mathbb{R}
\times \omega \text{,}
\end{equation*}%
$v$ is independent of $t$ in $%
\mathbb{R}
\times \omega $. Therefore, by using (3.20)$_{1}$, we get the following
equation%
\begin{equation*}
\Delta ^{2}v\left( t,x\right) -f_{1}\left( \left\Vert \nabla v\left(
t\right) \right\Vert _{L^{2}\left( \Omega \right) }\right) \Delta v\left(
t,x\right)
\end{equation*}%
\begin{equation}
+f_{2}\left( \left\Vert v\left( t\right) \right\Vert _{L^{2}\left( \Omega
\right) }\right) v\left( t,x\right) +\beta (x_{1})v\left( t,x\right) =0\text{%
, \ \ \ }\left( t,x\right) \in
\mathbb{R}
\times \omega \text{.}  \tag{3.22}
\end{equation}%
By writing the above equation at the point $t_{1}$ and $t_{2}$, subtracting
obtained equations and considering that $v$ is independent of $t$ in $%
\mathbb{R}
\times \omega $ , we get
\begin{equation*}
-\left( f_{1}\left( \left\Vert \nabla v\left( t_{2}\right) \right\Vert
_{L^{2}\left( \Omega \right) }\right) -f_{1}\left( \left\Vert \nabla v\left(
t_{1}\right) \right\Vert _{L^{2}\left( \Omega \right) }\right) \right)
\Delta v\left( t,x\right)
\end{equation*}%
\begin{equation*}
+\left( f_{2}\left( \left\Vert v\left( t_{2}\right) \right\Vert
_{L^{2}\left( \Omega \right) }\right) -f_{2}\left( \left\Vert v\left(
t_{1}\right) \right\Vert _{L^{2}\left( \Omega \right) }\right) \right)
v\left( t,x\right) =0\text{, \ \ }\left( t,x\right) \in
\mathbb{R}
\times \omega \text{.}
\end{equation*}%
Then we have
\begin{equation}
-\Delta v\left( t,x\right) +Cv\left( t,x\right) =0\text{, \ }\left(
t,x\right) \in
\mathbb{R}
\times \omega \text{,}  \tag{3.23}
\end{equation}%
where $C$ is the constant as follows
\begin{equation*}
C=\frac{f_{2}\left( \left\Vert v\left( t_{2}\right) \right\Vert
_{L^{2}\left( \Omega \right) }\right) -f_{2}\left( \left\Vert v\left(
t_{1}\right) \right\Vert _{L^{2}\left( \Omega \right) }\right) }{f_{1}\left(
\left\Vert \nabla v\left( t_{2}\right) \right\Vert _{L^{2}\left( \Omega
\right) }\right) -f_{1}\left( \left\Vert \nabla v\left( t_{1}\right)
\right\Vert _{L^{2}\left( \Omega \right) }\right) }\text{.}
\end{equation*}%
Applying Holmgren's theorem (see [1], [2]) to the equation (3.23), together
with boundary conditions (3.2), we find
\begin{equation*}
v\left( t,x\right) =0\text{, \ \ a.e. in \ }\omega \text{,}
\end{equation*}%
for all $t\in
\mathbb{R}
$. Then, by using extension outside of $%
\mathbb{R}
\times \Omega $ by zero, from (3.20), we obtain the problem%
\begin{equation*}
\left\{
\begin{array}{l}
\widetilde{v}_{tt}\left( t,x\right) +\Delta ^{2}\widetilde{v}\left(
t,x\right) -f_{1}\left( \left\Vert \nabla v\left( t\right) \right\Vert
_{L^{2}\left( \Omega \right) }\right) \Delta \widetilde{v}\left( t,x\right) +
\\
+f_{2}\left( \left\Vert v\left( t\right) \right\Vert _{L^{2}\left( \Omega
\right) }\right) \widetilde{v}\left( t,x\right) +\widetilde{\beta }(x_{1})%
\widetilde{v}\left( t,x\right) =0\text{,\ \ \ \ \ }\left( t,x\right) \in
\mathbb{R}
\times
\mathbb{R}
^{n}\text{,} \\
\widetilde{v}\left( t,x\right) =0\text{, \ \ \ \ \ }t\in
\mathbb{R}
\text{, \ }\left\vert x\right\vert \geq r\text{,}%
\end{array}%
\right.
\end{equation*}%
for some $r>0$, where $\widetilde{v}$ and $\widetilde{\beta }$ are the
extensions of $v$ and $\beta $. Since $\widetilde{v}\in C_{b}\left( \left[
0,\infty \right) ,H^{2}\left(
\mathbb{R}
^{n}\right) \right) $, by Theorem 1.1, we obtain%
\begin{equation*}
\widetilde{v}\left( t,x\right) =0\text{, \ \ a.e. in }%
\mathbb{R}
^{n}\text{,}
\end{equation*}%
for all $t\in
\mathbb{R}
$. But, then the term $f_{1}\left( \left\Vert \nabla v\left( t\right)
\right\Vert _{L^{2}\left( \Omega \right) }\right) $ is constant, which is a
contradiction. So, our assumption is false and $f_{1}\left( \left\Vert
\nabla v\left( t\right) \right\Vert _{L^{2}\left( \Omega \right) }\right) $
must be constant.

Now, secondly, assume that the term $f_{2}\left( \left\Vert v\left( t\right)
\right\Vert _{L^{2}\left( \Omega \right) }\right) $ is not constant. Then \
there exist $t_{1}$, $t_{2}\in
\mathbb{R}
$ such that $f_{2}\left( \left\Vert v\left( t_{1}\right) \right\Vert
_{L^{2}\left( \Omega \right) }\right) \neq $ $f_{2}\left( \left\Vert v\left(
t_{2}\right) \right\Vert _{L^{2}\left( \Omega \right) }\right) $. Therefore,
by using (3.22), considering that $v$ is independent of $t$ in $%
\mathbb{R}
\times \omega $ and $f_{1}\left( \left\Vert \nabla v\left( t\right)
\right\Vert _{L^{2}\left( \Omega \right) }\right) $ is constant, we get the
following equation
\begin{equation*}
\left( f_{2}\left( \left\Vert v\left( t_{2}\right) \right\Vert _{L^{2}\left(
\Omega \right) }\right) -f_{2}\left( \left\Vert v\left( t_{1}\right)
\right\Vert _{L^{2}\left( \Omega \right) }\right) \right) v\left( t,x\right)
=0\text{, \ \ }(t,x)\in
\mathbb{R}
\times \omega \text{.}
\end{equation*}%
Then we have%
\begin{equation*}
v\left( t,x\right) =0\text{, \ \ a.e. in }\omega \text{,}
\end{equation*}%
for all $t\in
\mathbb{R}
$. Thus, by using extension outside of $%
\mathbb{R}
\times \Omega $ by zero, from (3.20), we obtain the problem%
\begin{equation*}
\left\{
\begin{array}{l}
\widetilde{v}_{tt}\left( t,x\right) +\Delta ^{2}\widetilde{v}\left(
t,x\right) -f_{1}\left( \left\Vert \nabla v\left( t\right) \right\Vert
_{L^{2}\left( \Omega \right) }\right) \Delta \widetilde{v}\left( t,x\right) +
\\
+f_{2}\left( \left\Vert v\left( t\right) \right\Vert _{L^{2}\left( \Omega
\right) }\right) \widetilde{v}\left( t,x\right) +\widetilde{\beta }(x_{1})%
\widetilde{v}\left( t,x\right) =0\text{, \ \ \ \ }\left( t,x\right) \in
\mathbb{R}
\times
\mathbb{R}
^{n}\text{,} \\
\widetilde{v}\left( t,x\right) =0\text{, \ \ \ \ }t\in
\mathbb{R}
\text{, \ \ }\left\vert x\right\vert \geq r\text{,}%
\end{array}%
\right.
\end{equation*}%
for some $r>0$, where $\widetilde{v}$ and $\widetilde{\beta }$ are the
extensions of $v$ and $\beta $. Similarly, by applying Theorem 1.1, we obtain%
\begin{equation*}
\widetilde{v}\left( t,x\right) =0\text{, \ \ a.e. in }%
\mathbb{R}
^{n}\text{,}
\end{equation*}%
for all $t\in
\mathbb{R}
$. Then, it follows that $f_{2}\left( \left\Vert v\left( t\right)
\right\Vert _{L^{2}\left( \Omega \right) }\right) $ is constant, which
contradicts our assumption. So, our assumption is false and $f_{2}\left(
\left\Vert v\left( t\right) \right\Vert _{L^{2}\left( \Omega \right)
}\right) $ must be constant.

Consequently, $f_{1}\left( \left\Vert \nabla v\left( t\right) \right\Vert
_{L^{2}\left( \Omega \right) }\right) $ and $f_{2}\left( \left\Vert v\left(
t\right) \right\Vert _{L^{2}\left( \Omega \right) }\right) $ must be
constants. Let \newline
$f_{1}\left( \left\Vert \nabla v\left( t\right) \right\Vert _{L^{2}\left(
\Omega \right) }\right) \equiv c_{1}\geq 0$ and $f_{2}\left( \left\Vert
v\left( t\right) \right\Vert _{L^{2}\left( \Omega \right) }\right) \equiv
c_{2}\geq 0$. Then, by (3.20), we have%
\begin{equation*}
\left\{
\begin{array}{c}
v_{tt}\left( t,x\right) +\Delta ^{2}v\left( t,x\right) -c_{1}\Delta v\left(
t,x\right) +c_{2}v\left( t,x\right) +\beta (x_{1})v\left( t,x\right) =0\text{%
, }\left( t,x\right) \in
\mathbb{R}
\times \Omega \text{,} \\
v_{t}\left( t,x\right) =0\text{, \ \ }\left( t,x\right) \in
\mathbb{R}
\times \omega \text{.}%
\end{array}%
\right.
\end{equation*}%
Denoting $w^{h}=v(t+h)-v(t)$, from the above problem, we find%
\begin{equation*}
\left\{
\begin{array}{c}
w_{tt}^{h}\left( t,x\right) +\Delta ^{2}w^{h}\left( t,x\right) -c_{1}\Delta
w^{h}\left( t,x\right) +c_{2}w^{h}\left( t,x\right) +\beta
(x_{1})w^{h}\left( t,x\right) =0\text{, \ \ }(t,x)\in
\mathbb{R}
\times \Omega \text{,} \\
w^{h}\left( t,x\right) =0\text{, \ \ }\left( t,x\right) \in
\mathbb{R}
\times \omega \text{.}%
\end{array}%
\right.
\end{equation*}%
Here, by using extension outside of $%
\mathbb{R}
\times \Omega $ by zero, from the last problem, we get
\begin{equation*}
\left\{
\begin{array}{c}
\widetilde{w}_{tt}^{h}\left( t,x\right) +\Delta ^{2}\widetilde{w}^{h}\left(
t,x\right) -c_{1}\Delta \widetilde{w}^{h}\left( t,x\right) +c_{2}\widetilde{w%
}^{h}\left( t,x\right) +\widetilde{\beta }(x_{1})\widetilde{w}^{h}\left(
t,x\right) =0\text{, \ \ }\left( t,x\right) \in
\mathbb{R}
\times
\mathbb{R}
^{n}\text{,} \\
\widetilde{w}^{h}\left( t,x\right) =0\text{, \ \ }t\in
\mathbb{R}
\text{, }\left\vert x\right\vert \geq r\text{,}%
\end{array}%
\right.
\end{equation*}%
for some $r>0$, where $\widetilde{w}^{h}$ and $\widetilde{\beta }$ are the
extensions of $w^{h}$ and $\beta $. Since $\widetilde{w}^{h}\in C_{b}\left(
[0,\infty ),H^{2}\left(
\mathbb{R}
^{n}\right) \right) $, from Theorem 1.1, it follows that%
\begin{equation*}
\widetilde{w}^{h}\left( t,x\right) =0\text{, \ a.e. in }%
\mathbb{R}
^{n}\text{,}
\end{equation*}%
for all $t,h\in
\mathbb{R}
$, which yields (3.21). Then we have
\begin{equation*}
S\left( t\right) \varphi =\varphi \text{ , \ }\forall t\in
\mathbb{R}
\text{.}
\end{equation*}%
As a consequence, we obtain $\omega (\theta )\subset \mathcal{N}$ and by
(3.17), the proof of the lemma is complete.\newline
\end{proof}

\ Now, we can prove the main theorem of this section.

\begin{theorem}
Under conditions (3.4)-(3.7), the semigroup $\left\{ S\left( t\right)
\right\} _{t\geq 0}$ generated by (3.1)-(3.3) possesses a global attractor $%
\mathcal{A}$ in $H_{0}^{2}\left( \Omega \right) \times L^{2}\left( \Omega
\right) $ and $\mathcal{A=M}^{u}\left( \mathcal{N}\right) $. Here $\mathcal{M%
}^{u}\left( \mathcal{N}\right) $ is unstable manifold emanating from the set
of stationary points $\mathcal{N}$ (for definition, see [22, p.35]).
Moreover, the global attractor $\mathcal{A}$ consists of full trajectories $%
\gamma =\left\{ u\left( t\right) :t\in
\mathbb{R}
\right\} $ such that
\begin{equation}
\left\{
\begin{array}{c}
\underset{t\rightarrow -\infty }{\lim }\underset{v\in \mathcal{N}}{\inf }%
\left\Vert u\left( t\right) -v\right\Vert _{H_{0}^{2}(\Omega )\times
L^{2}(\Omega )}=0, \\
\underset{t\rightarrow \infty }{\lim }\underset{v\in \mathcal{N}}{\inf }%
\left\Vert u\left( t\right) -v\right\Vert _{H_{0}^{2}\left( \Omega \right)
\times L^{2}\left( \Omega \right) }=0\text{\ }%
\end{array}%
\right. \text{\ .}  \tag{3.24}
\end{equation}
\end{theorem}

\begin{proof}
From Lemma 3.1, Lemma 3.2 and [19, Theorem 3.1], it follows that there
exists a global attractor $\mathcal{A}$ in $H_{0}^{2}\left( \Omega \right)
\times L^{2}\left( \Omega \right) $. Then $\mathcal{M}^{u}\left( \mathcal{N}%
\right) \subset \mathcal{A}$ (see [22, Lemma 3.2, p.36]). Now, we prove the
other side of the inclusion. At first, as we discussed in the proof of Lemma
3.2, the semigroup $\left\{ S\left( t\right) \right\} _{t\geq 0}$ can be
extended to group $\left\{ S\left( t\right) \right\} _{t\in
\mathbb{R}
}$ . Let $\theta \in $ $\mathcal{A}$ be any point. Then it follows that the $%
\alpha $-limit set of $\theta $, namely,%
\begin{equation*}
\alpha \left( \theta \right) =\underset{t\leq 0}{\cap }\overline{\underset{%
\tau \leq t}{\cup }S\left( \tau \right) \theta }
\end{equation*}%
is nonempty, compact in $H_{0}^{2}\left( \Omega \right) \times L^{2}\left(
\Omega \right) $, invariant with respect to $S\left( t\right) $ and%
\begin{equation*}
\underset{t\rightarrow -\infty }{\lim }\underset{\phi \in \alpha \left(
\theta \right) }{\inf }\left\Vert S\left( t\right) \theta -\phi \right\Vert
_{H_{0}^{2}\left( \Omega \right) \times L^{2}\left( \Omega \right) }=0\text{.%
}
\end{equation*}%
By using the idea of the proof of Lemma 3.2, we obtain $\alpha (\theta
)\subset \mathcal{N}$ and hence, from the last equality,%
\begin{equation*}
\lim_{t\rightarrow -\infty }\inf_{v\in \mathcal{N}}\left\Vert S\left(
t\right) \theta -v\right\Vert _{H_{0}^{2}(\Omega )\times L^{2}(\Omega )}=0%
\text{, \ }\forall \theta \in \mathcal{A}\text{,}
\end{equation*}%
which gives (3.24)$_{1}$. Therefore, from the definition of unstable
manifold, we get
\begin{equation*}
\mathcal{A\subset M}^{u}\left( \mathcal{N}\right) \text{,}
\end{equation*}%
and consequently%
\begin{equation*}
\mathcal{A}=\mathcal{M}^{u}\left( \mathcal{N}\right) \text{.}
\end{equation*}%
Moreover, by Lemma 3.2, we have (3.24)$_{2}$ and hence, the proof is
complete.
\end{proof}

\end{document}